\documentclass[12pt,reqno]{amsart}

\usepackage{palatino}
\usepackage{amssymb}
\usepackage{amsmath}
\usepackage{enumitem}
\usepackage{geometry}
\usepackage{graphicx}
\usepackage{microtype}
\usepackage{xcolor}
\newtheorem{theorem}{Theorem}[section]
\newtheorem{proposition}[theorem]{Proposition}

\newtheorem{lemma}[theorem]{Lemma}

\newtheorem{remark}[theorem]{Remark}

\numberwithin{equation}{section}

\begin{document}

\title[Products of composition operators]
{Some Properties of Products of Linear Fractional Composition Operators on Weighted Dirichlet Spaces}

\author{Li He}
\address{Li He: School of Mathematics and Information Science, Guangzhou University, Guangzhou 510006, China.}
\email{helichangsha1986@163.com}

\author{Yuanhao Yan*}
\address{Yuanhao Yan: School of Mathematics and Information Science, Guangzhou University, Guangzhou 510006, China.}
\email{18322912287@163.com}

\author{Shuqing Zhang*}
\address{Shuqing Zhang: School of Mathematics and Information Science, Guangzhou University, Guangzhou 510006, China.}
\email{zsq0225@163.com}

\author{Xianfeng Zhao}
\address{Xianfeng Zhao: College of Mathematics and Statistics, Chongqing University, Chongqing, 401331, China}
\email{xianfengzhao@cqu.edu.cn}

	
\thanks{This work was supported by the National Natural Science Foundation of China (12371125, 12371127).}
\thanks{*Corresponding authors}	

\begin{abstract}
We study the Schatten class membership and essential norms of the operators $C_\varphi C_\psi^*$ and $C_\psi^*C_\varphi$, where $C_\varphi$ and $C_\psi$ are composition operators with nonconstant linear fractional symbols, acting on weighted Dirichlet spaces $\mathcal{D}_\alpha$ with $\alpha>-1$.
We obtain complete characterizations of their Schatten class membership in
terms of the boundary behavior of $\varphi$ and $\psi$. In particular, the
corresponding geometric conditions are independent of the Schatten exponent,
and each product belongs to every Schatten class whenever its boundary
condition is satisfied. We also investigate the essential norms of these
products. For $\alpha=0$ and $\alpha>0$, we obtain exact formulas in terms of
boundary contact and the derivatives of the symbols, while for
$-1<\alpha<0$ we obtain two-sided estimates by reducing the problem to
positive Toeplitz operators and local averages of generalized Nevanlinna
counting functions.
\end{abstract}

\maketitle

\section{Introduction}

Let $\mathbb D$ be the open unit disk in the complex plane $\mathbb C$, and
$\mathbb T$ be the unit circle. Let $dA$ denote the normalized area measure on
$\mathbb D$. For $\alpha>-1$, define
\[
dA_\alpha(z)
=
(1+\alpha)(1-|z|^2)^\alpha dA(z).
\]
For a measurable set $E\subset\mathbb D$, write
$A_\alpha(E)=\int_E dA_\alpha$.
The weighted Dirichlet space $\mathcal D_\alpha$ consists of the analytic
functions $f$ on $\mathbb D$ for which
\[
\|f\|_{\mathcal D_\alpha}^2
=
|f(0)|^2+
\int_{\mathbb D}|f'(z)|^2\,dA_\alpha(z)
<\infty.
\]
Let $\varphi$ be a nonconstant analytic self-map of $\mathbb D$. The
composition operator induced by $\varphi$ is defined by
\[
C_\varphi f=f\circ\varphi.
\]
Unlike the Hardy and Bergman spaces, not every analytic self-map induces a
bounded composition operator on Dirichlet-type space. For the nonconstant
linear fractional self-maps considered throughout this paper, however,
boundedness on $\mathcal D_\alpha$ is automatic; this also follows from the
derivative model recalled in Section~2.

A central theme in the theory of composition operators is to relate
operator theoretic properties of $C_\varphi$ to function theoretic and
geometric properties of its symbol. Compactness, essential norms, adjoint
formulas, and Schatten class membership have been studied extensively on
analytic function spaces; see, for example,
\cite{CowenMacCluer1995,Shapiro1993,Zhu2007}. If $\mathcal H$ is a separable
Hilbert space and $T$ is compact on $\mathcal H$, let
\[
s_1(T)\geqslant s_2(T)\geqslant\cdots\geqslant0
\]
denote its singular values. For $0<p<\infty$, the Schatten class
$S_p(\mathcal H)$ consists of those compact operators satisfying
\[
\sum_{n=1}^\infty s_n(T)^p<\infty.
\]
Thus Schatten class membership is a quantitative refinement of compactness. Another quantitative invariant is the essential norm, which measures the
distance of a bounded operator from the compact operators and therefore
records the size of its noncompact part. We refer to \cite{Simon2005} for
background on Schatten ideals and to \cite{Zhu2007} for standard operator
theory on analytic function spaces.

For a single composition operator, the relation between boundary behavior and
compactness has a long history. On the Hardy and Bergman spaces, angular
derivatives, Nevanlinna counting functions, and Carleson-type conditions play
a central role; see, among others, \cite{MacCluerShapiro1986,Shapiro1993}.
The Schatten ideal problem for single composition operators on the classical
Hardy and Bergman spaces was characterized by Luecking and Zhu
\cite{LueckingZhu1992}. On weighted Dirichlet spaces, boundedness and
compactness were studied by Zorboska \cite{Zorboska1998}, while Pau and
P\'erez obtained essential norm estimates, Schatten--von Neumann criteria,
and a closed range characterization \cite{PauPerez2013}. Li, Lu and Yu later
obtained exact essential norm formulas for composition operators in the
nonnegative range relevant to Section~5 \cite{LiLuYu2018}. The counting-function
criterion of Lef\`evre, Li, Queff\'elec and Rodr\'iguez-Piazza
\cite{LefevreEtAl2013}, recalled in Theorem~\ref{thm2.2}, characterizes
Schatten class membership on the weighted Dirichlet scale. When $\alpha=0$
and the symbol is univalent, it reduces to an image-area condition on
Hastings--Luecking windows. Related work on contact sets,
approximation numbers, and more general weighted analytic spaces can be found
in \cite{BourassEtAl2023,LefevreEtAl2013,LefevreEtAl2015}.

The situation becomes subtler for products involving an adjoint. Given two
analytic self-maps $\varphi$ and $\psi$, one is naturally led to the two
operators
\(
C_\varphi C_\psi^*
\text{~and~}
C_\psi^*C_\varphi.
\)

Although these products look formally symmetric, their behavior is generally
different. Clifford and Zheng initiated a systematic study of their
compactness on the Hardy space \cite{CliffordZheng1999} and later considered
the corresponding Bergman space problem \cite{CliffordZheng2003}. For
nonconstant linear fractional symbols, the two orders are controlled by two
different boundary geometries: the product $C_\varphi C_\psi^*$ is compact
precisely when $\varphi$ and $\psi$ have no common boundary value on
$\mathbb T$, whereas the compactness of $C_\psi^*C_\varphi$ is governed by
the absence of a common boundary contact point in the domain. In particular,
a product can be compact even when neither factor is compact, and compactness
can depend on the order of the factors. Further aspects of these products,
including commutators and invertibility, were investigated in
\cite{CliffordLeWiggins2014,CliffordLeviNarayan2012}.

The corresponding compactness problem on the classical Dirichlet space was
studied by G.~A. Chac\'on and G.~R. Chac\'on \cite{ChaconChacon2005}, who obtained complete
characterizations for linear fractional symbols. Their work is closely
connected with the special adjoint structure of linear fractional composition
operators. Cowen's formula on the Hardy space $H^2$ expresses the adjoint of
a linear fractional composition operator in terms of the Krein adjoint and
multiplication operators \cite{Cowen1988}. Hurst extended this formula to weighted Bergman
spaces \cite{Hurst1997}; Gallardo-Guti\'errez and Montes-Rodr\'iguez obtained an
adjoint formula on the classical Dirichlet space \cite{GallardoMontes2003};
and \v{C}u\v{c}kovi\'c and Le established a Cowen-type formula modulo compact
operators on a broad class of weighted Hardy spaces \cite{CuckovicLe2016}.
These formulas are particularly effective for products with an adjoint,
since they replace the adjoint factor by a composition operator induced by a
related linear fractional map together with analytic multipliers.

The weighted Dirichlet spaces provide a natural scale connecting several
classical Hilbert spaces of analytic functions, including the Hardy space,
the classical Dirichlet space, and weighted Bergman spaces. However, the
extension of operator-theoretic results from these endpoint spaces to the
whole weighted Dirichlet scale is not automatic. The main obstruction comes
from the lack of a direct adjoint formula for composition operators on
$\mathcal D_\alpha$. Although linear fractional symbols remain sufficiently
rigid to allow explicit boundary analysis, the corresponding products with
adjoints require a different approach from the Hardy and Bergman settings.

More recently, Leng and Zhao considered the Schatten-class version of the
product problem on the Hardy and Bergman spaces \cite{LengZhao2026}. In the
linear fractional case, they obtained complete characterizations on the Hardy
space for both orders of the product and for all $0<p<\infty$. On the Bergman
space, their forward product result again holds for every $0<p<\infty$, while
the corresponding reverse product theorem is obtained for $1<p<\infty$.
These developments naturally lead to the weighted Dirichlet setting,
where the interaction between boundary geometry and Schatten behavior can
be studied on the full scale $\alpha>-1$. The main problem addressed in this paper is to characterize the Schatten
class membership of these products in terms of the boundary geometry of
their symbols.

Unlike Hardy and Bergman spaces, the adjoint of a composition operator on weighted Dirichlet spaces does not admit a direct change of variable representation. Therefore, new arguments involving the derivative structure of the space and Toeplitz operator techniques are required. The first main result of this paper is a complete geometric characterization of the
Schatten class membership of both products. In contrast with the single
composition operator problem, the two product orders are not interchangeable and lead to different
boundary geometries. The forward product is governed by common boundary values of the images,
while the reverse product is determined by common boundary contact points
of the symbols. Specifically, for the forward product, Theorem~\ref{thm3.4} shows that,
for every $0<p<\infty$,
\[
C_\varphi C_\psi^*\in S_p(\mathcal D_\alpha)
\quad\Longleftrightarrow\quad
\varphi(\mathbb T)\cap\psi(\mathbb T)\cap\mathbb T=\varnothing.
\]
For the reverse product, Theorem~\ref{thm4.8} gives
\[
C_\psi^*C_\varphi\in S_p(\mathcal D_\alpha)
\quad\Longleftrightarrow\quad
\varphi^{-1}(\mathbb T)\cap\psi^{-1}(\mathbb T)\cap\mathbb T
=\varnothing.
\]
When $\alpha=0$, these conditions reduce to the compactness criteria in \cite{ChaconChacon2005}. The present results extend the
same boundary geometries to all weighted Dirichlet spaces and show, moreover,
that whenever the relevant boundary separation holds, the product belongs to
every Schatten class $S_q$ with $q>0$.

Beyond Schatten membership, we study the essential norms of these products
and show how the same boundary geometry controls both compactness and
quantitative noncompactness. On the classical Dirichlet space,
Theorems~\ref{thm5.2} and \ref{thm5.3} show that the essential norm is either
$0$ or $1$: it vanishes exactly under the corresponding compactness condition
and equals $1$ otherwise. For $\alpha>0$, we obtain exact formulas. More precisely, Theorems \ref{thm5.10} and \ref{thm5.11} give
\[
\|C_\varphi C_\psi^*\|_{e,\mathcal D_\alpha}^2
=
\max_{\substack{\xi,\zeta\in\mathbb T\\
\varphi(\xi)=\psi(\zeta)\in\mathbb T}}
\bigl(|\varphi'(\xi)|\,|\psi'(\zeta)|\bigr)^{-\alpha}
\]
and
\[
\|C_\psi^*C_\varphi\|_{e,\mathcal D_\alpha}^2
=
\max_{\substack{\lambda\in\mathbb T\\
|\varphi(\lambda)|=|\psi(\lambda)|=1}}
\bigl(|\varphi'(\lambda)|\,|\psi'(\lambda)|\bigr)^{-\alpha},
\]
where each maximum is understood to be zero when its indexing set is empty.
Thus the same boundary geometry that decides compactness also determines the
exact size of the essential norm.

For $-1<\alpha<0$, the boundary formulas available in the positive range are replaced by
two-sided estimates involving positive Toeplitz operators and local averages of generalized Nevanlinna
counting functions. Let $\tau$ denote the
Krein adjoint of $\psi$. For every fixed $r>0$,
Theorems~\ref{thm5.13} and \ref{thm5.14} yield
\[
\|C_\varphi C_\psi^*\|_{e,\mathcal D_\alpha}^2
\asymp
\limsup_{|z|\to1^-}
\frac{\displaystyle
\int_{D(z,r)\cap(\tau\circ\varphi)(\mathbb D)}
N_{\tau\circ\varphi,\alpha}(w)\,dA(w)}
{A_\alpha(D(z,r))}
\]
and
\[
\|C_\psi^*C_\varphi\|_{e,\mathcal D_\alpha}^2
\asymp
\limsup_{|z|\to1^-}
\frac{\displaystyle
\int_{D(z,r)\cap(\varphi\circ\tau)(\mathbb D)}
N_{\varphi\circ\tau,\alpha}(w)\,dA(w)}
{A_\alpha(D(z,r))}.
\]
Here and below, $X\asymp Y$ means that there exist positive constants $c$ and $C$, independent of the essential variable, such that $cY\leqslant X\leqslant CY$. The implicit constants are independent of $z$. These estimates are obtained
by reducing the problem to the essential norm of a positive Toeplitz operator
and then applying Yamaji's local-average estimate \cite{Yamaji2013}. This
approach is well suited to the singular weight and complements the exact
formulas available in the nonnegative range.

The proofs also explain why the two product orders must be treated
separately. For the forward product, the exact Cowen--Hurst adjoint
formula on $A_\alpha^2$ gives the factorization
\[
D_\varphi D_\psi^*
=
M_uD_{\tau\circ\varphi}M_v^*,
\]
where $M_u$ and $M_v$ are bounded invertible multiplication operators.
This reduces the Schatten and essential norm problems to those of a single
linear fractional symbol. For the reverse product, the corresponding
reduction is available only after passing modulo compact operators and using
the Cowen-type adjoint formula on $\mathcal D_\alpha$. These different
reductions lead naturally to the two distinct boundary conditions appearing
in the Schatten characterizations. In the reverse case, the Schatten
membership is obtained by finite rank approximations with exponentially
decaying tails.

The paper is organized as follows. Section~2 contains the preliminaries,
including the weighted Bergman derivative model, the Schatten criterion for a
single composition operator, the Cowen--Hurst adjoint formula, and the basic
essential norm and Toeplitz tools used later. Section~3 studies the Schatten
class membership of $C_\varphi C_\psi^*$, and Section~4 treats the reverse
product $C_\psi^*C_\varphi$. Section~5 develops the essential norm theory. The classical Dirichlet case is treated first, followed by the positive
range where explicit boundary formulas are available and the negative
range where a Toeplitz-theoretic approach is required.

\section{Preliminaries}

We collect here the tools needed in the proofs. The basic definitions of
$\mathcal D_\alpha$ and $S_p$ were given in the Introduction. We shall use the
weighted Bergman derivative model to pass between $\mathcal D_\alpha$ and
$A_\alpha^2$, a Schatten criterion for a single composition operator, the
Cowen--Hurst adjoint formula for linear fractional symbols, and several
standard facts about essential norms and Toeplitz operators. These ingredients
will be used repeatedly in Sections~3--5.

\subsection{Weighted Bergman spaces and the derivative model}\mbox{}

With $dA_\alpha$ as above, the standard weighted Bergman space $A_\alpha^2$
consists of analytic functions $g$ on $\mathbb D$ such that
\[
\|g\|_{A_\alpha^2}^2
=
\int_{\mathbb D}|g(z)|^2\,dA_\alpha(z)
<\infty.
\]
The reproducing kernel  for $A_\alpha^2$  is given by
\[
K_z^\alpha(w)=\frac{1}{(1-\overline zw)^{\alpha+2}},
\]
and we write
\[
k_z^\alpha
=
\frac{K_z^\alpha}{\|K_z^\alpha\|_{A_\alpha^2}}
\]
for the normalized reproducing kernel.
Let
\[
\mathcal D_{\alpha,0}
=
\{f\in\mathcal D_\alpha:f(0)=0\}\ \ \ \ \mathrm{and} \ \ \ \
\mathcal D_\alpha
=
\mathbb C\oplus\mathcal D_{\alpha,0}.
\]

For an analytic self-map $\varphi$ of $\mathbb D$, define the normalized
composition operator on $\mathcal D_{\alpha,0}$ by
\[
\widetilde C_\varphi f
=
f\circ\varphi-f(\varphi(0)).
\]
Moreover, we define
\[
D_\varphi g
=
(g\circ\varphi)\varphi',
\qquad g\in A_\alpha^2.
\]

\begin{proposition}\label{prop2.1}
Let $\alpha>-1$, and let $\varphi$ be an analytic self-map of $\mathbb D$
for which $C_\varphi$ is bounded on $\mathcal D_\alpha$. Define
\[
U_\alpha:\mathcal D_{\alpha,0}\longrightarrow A_\alpha^2,
\qquad
U_\alpha f=f'.
\]
Then $U_\alpha$ is unitary and
\[
U_\alpha\widetilde C_\varphi U_\alpha^*
=
D_\varphi.
\]
Moreover, relative to
$\mathcal D_\alpha=\mathbb C\oplus\mathcal D_{\alpha,0}$,
\[
C_\varphi
=
\begin{pmatrix}
I_{\mathbb C}&\Lambda_{\varphi(0)}\vspace{2mm}\\
0&\widetilde C_\varphi
\end{pmatrix},
\qquad
\Lambda_{\varphi(0)}f=f(\varphi(0)).
\]
Consequently, for every $0<p<\infty$,
\[
C_\varphi\in S_p(\mathcal D_\alpha)
\quad\Longleftrightarrow\quad
D_\varphi\in S_p(A_\alpha^2).
\]
Furthermore,
\[
\|C_\varphi\|_{e,\mathcal D_\alpha}
=
\|\widetilde C_\varphi\|_{e,\mathcal D_{\alpha,0}}
=
\|D_\varphi\|_{e,A_\alpha^2}.
\]
\end{proposition}

\begin{proof}
For $f\in\mathcal D_{\alpha,0}$,
\[
\|U_\alpha f\|_{A_\alpha^2}^2
=
\int_{\mathbb D}|f'(z)|^2\,dA_\alpha(z)
=
\|f\|_{\mathcal D_\alpha}^2,
\]
which implies that $U_\alpha$ is isometric. Moreover, for every $g\in A_\alpha^2$,
\[
(U_\alpha^*g)(z)
=
\int_0^z g(\zeta)\,d\zeta,
\]
and hence
\[
U_\alpha U_\alpha^*g=g.
\]
Thus $U_\alpha$ is onto, and therefore unitary.
\begin{align*}
(U_\alpha\widetilde C_\varphi U_\alpha^*g)(z)
&=
\frac{d}{dz}
\left[
(U_\alpha^*g)(\varphi(z))
-
(U_\alpha^*g)(\varphi(0))
\right]
\\
&=
(U_\alpha^*g)'(\varphi(z))\varphi'(z)
\\
&=
g(\varphi(z))\varphi'(z)
\\
&=
D_\varphi g(z).
\end{align*}
Consequently,
\[
U_\alpha\widetilde C_\varphi U_\alpha^*
=
D_\varphi.
\tag{2.1}
\]

Now consider the full weighted Dirichlet space. Every
$f\in\mathcal D_\alpha$ can be written uniquely as
\[
f=c+f_0,
\qquad
c\in\mathbb C,\quad
f_0\in\mathcal D_{\alpha,0}.
\]
Then
\begin{align*}
C_\varphi f
&=
c+f_0\circ\varphi
\\
&=
c+f_0(\varphi(0))
+
\bigl(
f_0\circ\varphi-f_0(\varphi(0))
\bigr)
\\
&=
c+\Lambda_{\varphi(0)}f_0
+
\widetilde C_\varphi f_0.
\end{align*}
Since point evaluation at $\varphi(0)$ is bounded on
$\mathcal D_{\alpha,0}$, the functional
\[
\Lambda_{\varphi(0)}f_0
=
f_0(\varphi(0))
\]
is bounded. Hence, relative to the orthogonal decomposition
\[
\mathcal D_\alpha
=
\mathbb C\oplus\mathcal D_{\alpha,0},
\]
we obtain
\[
C_\varphi
=
\begin{pmatrix}
I_{\mathbb C} & \Lambda_{\varphi(0)}
\vspace{2mm}\\
0 & \widetilde C_\varphi
\end{pmatrix}.
\tag{2.2}
\]

Let
\[
F_\varphi
=
C_\varphi-
\begin{pmatrix}
0&0\vspace{2mm}\\
0&\widetilde C_\varphi
\end{pmatrix}
=
\begin{pmatrix}
I_{\mathbb C}&\Lambda_{\varphi(0)}
\vspace{2mm}\\
0&0
\end{pmatrix}.
\]
The range of $F_\varphi$ is contained in
$\mathbb C\oplus\{0\}$, and hence
\[
\operatorname{rank}(F_\varphi) \leqslant 1.
\]
Thus $F_\varphi$ belongs to $S_p(\mathcal D_\alpha)$ for every
$0<p<\infty$. Since $S_p$ is a linear ideal, it follows that
\[
C_\varphi\in S_p(\mathcal D_\alpha)
\quad\Longleftrightarrow\quad
0\oplus\widetilde C_\varphi
\in S_p(\mathcal D_\alpha).
\]
Equivalently,
\[
C_\varphi\in S_p(\mathcal D_\alpha)
\quad\Longleftrightarrow\quad
\widetilde C_\varphi
\in S_p(\mathcal D_{\alpha,0}).
\tag{2.3}
\]

Finally, by \textup{(2.1)}, $\widetilde C_\varphi$ and $D_\varphi$
are unitarily equivalent. Therefore, they have the same singular values,
and hence
\[
\widetilde C_\varphi
\in S_p(\mathcal D_{\alpha,0})
\quad\Longleftrightarrow\quad
D_\varphi
\in S_p(A_\alpha^2).
\tag{2.4}
\]
Combining \textup{(2.3)} and \textup{(2.4)}, we conclude that
\[
C_\varphi\in S_p(\mathcal D_\alpha)
\quad\Longleftrightarrow\quad
D_\varphi\in S_p(A_\alpha^2).
\]
Since $F_\varphi$ has finite rank, the block representation also gives
\[
\|C_\varphi\|_{e,\mathcal D_\alpha}
=
\|\widetilde C_\varphi\|_{e,\mathcal D_{\alpha,0}}.
\]
The unitary equivalence in \textup{(2.1)} then yields
\[
\|\widetilde C_\varphi\|_{e,\mathcal D_{\alpha,0}}
=
\|D_\varphi\|_{e,A_\alpha^2}.
\]
This completes the proof.
\end{proof}

If $\varphi$ is a linear fractional self-map of $\mathbb D$, then
$C_\varphi$ is automatically bounded on $\mathcal D_\alpha$. Indeed,
$\varphi'\in H^\infty$, while the composition operator induced by
$\varphi$ is bounded on $A_\alpha^2$. Hence
\[
D_\varphi=M_{\varphi'}C_\varphi
\]
is bounded on $A_\alpha^2$. The assertion then follows from the
unitary equivalence and the block representation above.

\subsection{A Schatten criterion on weighted Dirichlet spaces}\mbox{}

We first recall some elementary facts about Schatten classes that will be used
throughout the paper. If $T\in S_p(\mathcal H)$ and
$A,B\in\mathcal B(\mathcal H)$, then
\[
ATB\in S_p(\mathcal H).
\]
In particular, if $A$ and $B$ are bounded and invertible, then
\[
T\in S_p(\mathcal H)
\quad\Longleftrightarrow\quad
ATB\in S_p(\mathcal H).
\]
Moreover, every finite rank operator belongs to $S_p$ for every $p>0$;
see \cite{Simon2005}.

For an analytic self-map $\varphi$ of $\mathbb D$, we denote by
\[
N_{\varphi,\alpha}(w)
=
\sum_{\varphi(z)=w}(1-|z|^2)^\alpha
\]
the generalized Nevanlinna counting function, where the preimages are counted
according to multiplicity. We set $N_{\varphi,\alpha}(w)=0$ whenever
$w\notin\varphi(\mathbb D)$.

For $n\geqslant0$ and $0\leqslant j\leqslant  2^n-1$, define the Hastings--Luecking windows by
\[
R_{n,j}
=
\left\{
z\in\mathbb D:
1-2^{-n}\leqslant |z|<1-2^{-n-1},
\quad
\frac{2j\pi}{2^n}\leqslant \arg (z)<
\frac{2(j+1)\pi}{2^n}
\right\}.
\]

We shall use the counting-function assertion of
\cite[Theorem 3.1]{LefevreEtAl2013}, formulated for the subspace
$\mathcal D_{\alpha,0}$ of functions vanishing at the origin.

\begin{theorem}
\label{thm2.2}
Let $\alpha>-1$, $0<p<\infty$, and let $\varphi$ be an analytic self-map
of $\mathbb D$. Then
\[
\widetilde C_\varphi\in S_p(\mathcal D_{\alpha,0})
\]
if and only if
\[
\sum_{n=0}^\infty
\sum_{j=0}^{2^n-1}
\left[
2^{n(\alpha+2)}
\int_{R_{n,j}}N_{\varphi,\alpha}(w)\,dA(w)
\right]^{p/2}
<\infty.
\]
If $\varphi$ is univalent and $\Omega=\varphi(\mathbb D)$, then the above
condition is equivalent to
\[
\sum_{n=0}^\infty
\sum_{j=0}^{2^n-1}
\left[
2^{n(\alpha+2)}
\int_{R_{n,j}\cap\Omega}
(1-|\varphi^{-1}(w)|^2)^\alpha\,dA(w)
\right]^{p/2}
<\infty.
\]
\end{theorem}

\begin{remark}
The counting-function criterion in Theorem~\ref{thm2.2} is based on
\cite[Theorem~3.1]{LefevreEtAl2013}. In the univalent case, the
weighted-area condition (3.6) in that reference uses
$A_\alpha(R_{n,j}\cap\Omega)$.
However, by the definition of the weighted counting function,
\[
N_{\varphi,\alpha}(w)
=(1-|\varphi^{-1}(w)|^2)^\alpha,
\qquad w\in\Omega.
\]
Thus the weight must be evaluated at the preimage, as in the
univalent condition of Theorem~\ref{thm2.2}. The two formulations
agree when $\alpha=0$, but the weighted area condition (3.6)
does not hold in general when $\alpha\ne0$.
\end{remark}
\subsection{The Cowen--Hurst adjoint formula}\mbox{}

Let
\[
\psi(z)
=
\frac{az+b}{cz+d}\ \ \ \
(\Delta=ad-bc \neq 0)
\]
be a linear fractional self-map of $\mathbb D$. Its Krein adjoint is
\[
\tau(z)
=
\frac{\overline a z-\overline c}
{-\overline b z+\overline d}.
\]
The following weighted Bergman adjoint formula is due to Hurst
\cite{Hurst1997}; see also \cite[Theorem~2.6]{CuckovicLe2016}.

\begin{theorem}\label{thm2.4}
Let $\alpha>-1$ and let $\psi$ and $\tau$ be as above. On $A_\alpha^2$,
\[
C_\psi^*
=
M_gC_\tau M_h^*,
\]
where
\[
g(z)
=
(-\overline b z+\overline d)^{-\alpha-2}\ \ \  \ \text{and} \ \ \ \
h(z)
=
(cz+d)^{\alpha+2}.
\]
The powers are taken with compatible analytic branches on a neighborhood of $\overline{\mathbb D}$, normalized so that $g(0)\overline{h(0)}=1$. This normalization ensures that the phase factors arising from the choice of branches cancel in the adjoint formula.
\end{theorem}

For a linear fractional self-map, the functions
$cz+d$ and $-\overline b z+\overline d$ have no zeros on
$\overline{\mathbb D}$. Hence the functions occurring in
Theorem~\ref{thm2.4}, as well as their reciprocals whenever
needed below, are bounded analytic functions on a neighborhood of
$\overline{\mathbb D}$.

\subsection{Essential norms and Toeplitz operators}\mbox{}

For a Hilbert space $\mathcal H$, let $\mathcal B(\mathcal H)$ and
$\mathcal K(\mathcal H)$ denote the set of bounded operators and the set of compact operators,
respectively.  Let
\[
\pi:\mathcal B(\mathcal H)
\longrightarrow
\mathcal B(\mathcal H)/\mathcal K(\mathcal H)
\]
be the quotient map. The essential norm of $T\in\mathcal B(\mathcal H)$ is defined by
\[
\|T\|_e
=
\inf_{K\in\mathcal K(\mathcal H)}\|T-K\|
=
\|\pi(T)\|.
\]
Since the Calkin algebra is a $C^*$-algebra,
\[
\|T\|_e^2
=
\|T^*T\|_e.
\]
We shall also use the ideal estimate
\[
\|ATB\|_e
\leqslant
\|A\|\,\|T\|_e\,\|B\|,
\qquad
A,B,T\in\mathcal B(\mathcal H).
\]
In particular, if $A$ and $B$ are bounded and invertible, then
\[
\|ATB\|_e\asymp\|T\|_e.
\]

Let $P_\alpha$ denote the Bergman projection from
$L^2(\mathbb D,dA_\alpha)$ onto $A_\alpha^2$. For a bounded measurable
symbol $a$, the Toeplitz and Hankel operators are defined by
\[
T_a=P_\alpha M_a|_{A_\alpha^2}\ \  \ \ \text{and} \ \ \ \
H_a=(I-P_\alpha)M_a|_{A_\alpha^2},
\]
respectively.
More generally, if $\mu$ is a positive Borel measure for which the embedding
$A_\alpha^2\to L^2(\mu)$ is bounded, the positive Toeplitz
operator $T_\mu$ is determined by
\[
\langle T_\mu f,g\rangle_{A_\alpha^2}
=
\int_{\mathbb D}f(w)\overline{g(w)}\,d\mu(w),
\qquad f,g\in A_\alpha^2.
\]

Let $z,w\in\mathbb D$. The pseudo-hyperbolic metric and Bergman metric are defined by
\[
\rho(z,w)
=
\left|\frac{z-w}{1-\overline zw}\right|\ \ \ \ \text{and}\ \ \ \
\beta(z,w)
=
\frac12\log\frac{1+\rho(z,w)}{1-\rho(z,w)},
\]
respectively.
For $r>0$, the hyperbolic disk is defined by
\[
D(z,r)=\{w\in\mathbb D:\beta(z,w)<r\}.
\]
A consequence of Yamaji's essential-norm estimate for positive Toeplitz
operators \cite[Theorem~B]{Yamaji2013} is that, for every fixed $r>0$,
\[
\|T_\mu\|_{e,A_\alpha^2}
\asymp
\limsup_{|z|\to1^-}
\frac{\mu(D(z,r))}{A_\alpha(D(z,r))},
\]
whenever $T_\mu$ is bounded on $A_\alpha^2$. To see the correspondence with
the notation in \cite{Yamaji2013}, one may set $\gamma=\alpha/2$ and
$dV_\gamma(z)=(1-|z|^2)^{2\gamma}dA(z)$. Then we have that
$dA_\alpha=(1+\alpha)dV_\gamma$, and the constant factor does not affect the
two-sided estimate. This form will be used in Section~5 for the range
$-1<\alpha<0$.

\section{The product $C_\varphi C_\psi^*$}

We now prove the Schatten class reduction that replaces the corresponding
Hardy-space lemma in \cite{LengZhao2026}.

\begin{lemma}\label{lem3.1}
Let $\alpha>-1$ and $0<p<\infty$. Let $\varphi$ and $\psi$ be
nonconstant linear fractional self-maps of $\mathbb D$, and let $\tau$
be the Krein adjoint of $\psi$. Then
\[
C_\varphi C_\psi^*
\in S_p(\mathcal D_\alpha)
\quad\Longleftrightarrow\quad
C_{\tau\circ\varphi}
\in S_p(\mathcal D_\alpha).
\]
\end{lemma}

\begin{proof}
By Proposition~\ref{prop2.1} and the block decomposition of
$C_\varphi C_\psi^*$,
\[
C_\varphi C_\psi^*
\in S_p(\mathcal D_\alpha)
\quad\Longleftrightarrow\quad
D_\varphi D_\psi^*
\in S_p(A_\alpha^2).
\]
Indeed, relative to
$\mathcal D_\alpha=\mathbb C\oplus\mathcal D_{\alpha,0}$,
\[
C_\varphi C_\psi^*
=
\begin{pmatrix}
I+\Lambda_{\varphi(0)}\Lambda_{\psi(0)}^*
&
\Lambda_{\varphi(0)}\widetilde C_\psi^*
\vspace{2mm}\\
\widetilde C_\varphi\Lambda_{\psi(0)}^*
&
\widetilde C_\varphi\widetilde C_\psi^*
\end{pmatrix},
\]
and all blocks except the lower-right one have finite rank. Then we obtain
\[
C_\varphi C_\psi^*
\in S_p(\mathcal D_\alpha)
\quad\Longleftrightarrow\quad
\widetilde C_\varphi\widetilde C_\psi^*
\in S_p(\mathcal D_{\alpha,0}).
\]

By Proposition~\ref{prop2.1},
\[
U_\alpha\widetilde C_\varphi U_\alpha^*
=
D_\varphi
\]
and
\[
U_\alpha\widetilde C_\psi U_\alpha^*
=
D_\psi.
\]
Taking adjoints in the second identity gives
\[
U_\alpha\widetilde C_\psi^*U_\alpha^*
=
D_\psi^*.
\]
Therefore,
\begin{align*}
U_\alpha
\widetilde C_\varphi\widetilde C_\psi^*
U_\alpha^*
&=
\bigl(U_\alpha\widetilde C_\varphi U_\alpha^*\bigr)
\bigl(U_\alpha\widetilde C_\psi^*U_\alpha^*\bigr)=
D_\varphi D_\psi^*.
\end{align*}
Since Schatten class membership is invariant under unitary equivalence,
\[
\widetilde C_\varphi\widetilde C_\psi^*
\in S_p(\mathcal D_{\alpha,0})
\quad\Longleftrightarrow\quad
D_\varphi D_\psi^*
\in S_p(A_\alpha^2).
\]
Consequently,
\[
C_\varphi C_\psi^*
\in S_p(\mathcal D_\alpha)
\quad\Longleftrightarrow\quad
D_\varphi D_\psi^*
\in S_p(A_\alpha^2).
\]
Since
\[
D_\psi=M_{\psi'}C_\psi
\]
on $A_\alpha^2$, Theorem~\ref{thm2.4} gives
\[
D_\psi^*
=
C_\psi^*M_{\psi'}^*
=
M_gC_\tau M_h^*M_{\psi'}^*
=
M_gC_\tau M_v^*,
\]
where
\[
v=\psi'h.
\]
Since
\[
\psi'(z)
=
\frac{\Delta}{(cz+d)^2},
\]
we have
\[
v(z)
=
\Delta(cz+d)^\alpha.
\]
Consequently,
\begin{align*}
D_\varphi D_\psi^*
&=
M_{\varphi'}C_\varphi M_gC_\tau M_v^*=
M_{\varphi'(g\circ\varphi)}
C_{\tau\circ\varphi}M_v^*.
\end{align*}

On the other hand,
\[
D_{\tau\circ\varphi}
=
M_{(\tau'\circ\varphi)\varphi'}
C_{\tau\circ\varphi}.
\]
A direct calculation gives
\[
\tau'(z)
=
\frac{\overline\Delta}
{(-\overline b z+\overline d)^2}.
\]
It follows that
\[
D_\varphi D_\psi^*
=
M_uD_{\tau\circ\varphi}M_v^*,
\]
where
\[
u(z)
=
\frac{g(\varphi(z))}
{\tau'(\varphi(z))}
=
\frac{1}{\overline\Delta}
\bigl(-\overline b\varphi(z)+\overline d\bigr)^{-\alpha}.
\]
Both $u$ and $v$, together with their reciprocals, belong to $H^\infty$.
Thus $M_u$ and $M_v$ are bounded and invertible on $A_\alpha^2$.
The ideal property of the Schatten classes yields
\[
D_\varphi D_\psi^*
\in S_p(A_\alpha^2)
\quad\Longleftrightarrow\quad
D_{\tau\circ\varphi}
\in S_p(A_\alpha^2).
\]
A final application of Proposition~\ref{prop2.1} gives the
desired equivalence.
\end{proof}

We next prove the Schatten dichotomy for a single linear fractional symbol using Theorem~\ref{thm2.2} and a local inverse estimate near a boundary contact point.

\begin{proposition}\label{prop3.2}
Let $\alpha>-1$, $0<p<\infty$, and let $\eta$ be a nonconstant linear
fractional self-map of $\mathbb D$. Then
\[
C_\eta\in S_p(\mathcal D_\alpha)
\quad\Longleftrightarrow\quad
\overline{\eta(\mathbb D)}\subset\mathbb D.
\]
If this condition holds, then
\[
C_\eta\in S_q(\mathcal D_\alpha)
\]
for every $0<q<\infty$.
\end{proposition}

\begin{proof}
Put $\Omega=\eta(\mathbb D)$. Since $\eta$ is univalent,
Proposition~\ref{prop2.1} and Theorem~\ref{thm2.2} give
\[
C_\eta\in S_p(\mathcal D_\alpha)
\quad\Longleftrightarrow\quad
\sum_{n=0}^\infty\sum_{j=0}^{2^n-1}
\left[
2^{n(\alpha+2)}
\int_{R_{n,j}\cap\Omega}
(1-|\eta^{-1}(w)|^2)^\alpha\,dA(w)
\right]^{p/2}
<\infty.
\tag{3.1}
\]

Suppose first that $\overline{\Omega}\subset\mathbb D$.
Then $\Omega\subset\rho\mathbb D$ for some $\rho<1$, so the
integrals in \textup{(3.1)} vanish for all sufficiently large $n$.
The remaining integrals are finite, since the change of variables
formula gives
\[
\int_\Omega(1-|\eta^{-1}(w)|^2)^\alpha\,dA(w)
=
\int_{\mathbb D}
(1-|z|^2)^\alpha|\eta'(z)|^2\,dA(z)
<\infty.
\]
Here $\eta'\in H^\infty$ and $\alpha>-1$. Thus
$C_\eta\in S_p(\mathcal D_\alpha)$ for every $p>0$.

Conversely, suppose that
$\overline{\Omega}\cap\mathbb T\ne\varnothing$.
Choose $\xi,\zeta\in\mathbb T$ such that $\eta(\xi)=\zeta$.
By the Julia--Carath\'eodory theorem,
\[
\lambda=\xi\eta'(\xi)\overline{\zeta}>0.
\]
For $h=2^{-n}$, let
\[
E_h=
\left\{
w\in\mathbb C:
\left|w-\zeta\left(1-\frac{3h}{4}\right)\right|
<\frac h4
\right\}.
\tag{3.2}
\]
For $w\in E_h$, the triangle inequality gives
\[
1-h<|w|<1-\frac h2.
\]
Moreover, if $h\leq 1/2$, then
\[
|\arg(\overline{\zeta}w)|
\leq
\arctan\frac{h/4}{1-h}
\leq \frac h2.
\]
Thus $E_h$ lies in the radial annulus of the $n$th
Hastings--Luecking level and meets at most two of its angular windows.

Write
\[
\overline{\zeta}w=1-s
\qquad(w\in E_h).
\]
Then
\[
\frac h2<\operatorname{Re}s<h,
\qquad |s|<h.
\]
The local inverse of $\eta$ is analytic near $\zeta$. Its Taylor
expansion at $\zeta$ gives, uniformly for $w\in E_h$,
\[
\eta^{-1}(w)
=
\xi\left(1-\frac{s}{\lambda}+O(h^2)\right).
\]
Consequently,
\[
1-|\eta^{-1}(w)|^2
=
\frac{2\operatorname{Re}s}{\lambda}+O(h^2)
\asymp h.
\]
In particular, $\eta^{-1}(w)\in\mathbb D$ for all sufficiently
small $h$, and hence $E_h\subset\Omega$. The two-sided estimate
also gives, for every $\alpha>-1$,
\[
(1-|\eta^{-1}(w)|^2)^\alpha\asymp h^\alpha
\qquad(w\in E_h).
\]

Since $dA$ is normalized area measure,
\[
A(E_h)=\left(\frac h4\right)^2=\frac{h^2}{16}.
\]
It follows that
\[
\int_{E_h}
(1-|\eta^{-1}(w)|^2)^\alpha\,dA(w)
\gtrsim h^{\alpha+2}.
\]
As $E_h$ meets at most two windows, there is an index $j_n$ such that
\[
2^{n(\alpha+2)}
\int_{R_{n,j_n}\cap\Omega}
(1-|\eta^{-1}(w)|^2)^\alpha\,dA(w)
\gtrsim 1
\]
for every sufficiently large $n$. Hence the series in
\textup{(3.1)} diverges for every $p>0$. This proves the converse.
\end{proof}
A key step in our approach is a reduction principle for linear fractional
symbols: on every weighted Dirichlet space, the Schatten membership of a
composition operator is equivalent to a geometric condition on the image
domain. This reduction allows us to transform the product problems into
boundary geometric questions.

\begin{lemma}\label{lem3.3}
Let $\varphi$ and $\psi$ be nonconstant linear fractional self-maps of
$\mathbb D$, and let $\tau$ be the Krein adjoint of $\psi$. Then
\[
\overline{(\tau\circ\varphi)(\mathbb D)}
\subset\mathbb D
\]
if and only if
\[
\varphi(\mathbb T)
\cap
\psi(\mathbb T)
\cap
\mathbb T
=
\varnothing.
\]
\end{lemma}

\begin{proof}
Define
\[
\delta(z)=\frac1z
\]
on the Riemann sphere $\widehat{\mathbb C}=\mathbb{C}\cup\{\infty\}$ and
\[
\widetilde\psi(z)
=
\overline{\psi(\overline z)}.
\]
A direct calculation gives
\[
\tau
=
\delta\circ(\widetilde\psi)^{-1}\circ\delta\ \ \ \ \text{and} \ \ \ \
\tau^{-1}
=
\delta\circ\widetilde\psi\circ\delta.
\tag{3.3}
\]
It follows that
\[
(\tau\circ\varphi)(\mathbb D)
=
(\delta\circ\widetilde\psi^{-1}\circ\delta\circ\varphi)(\mathbb D)
\subset
(\delta\circ\widetilde\psi^{-1})
(\widehat{\mathbb C}\backslash\overline{\mathbb D})
\subset
\delta(\widehat{\mathbb C}\backslash\overline{\mathbb D})
=
\mathbb D.
\]
Assume first that
\[
\overline{(\tau\circ\varphi)(\mathbb D)}
\not\subset\mathbb D.
\]

Since $\tau\circ\varphi$ is a linear fractional self-map and extends
continuously to $\overline{\mathbb D}$, there exist
$\xi,\zeta\in\mathbb T$ such that
\[
(\tau\circ\varphi)(\xi)=\zeta.
\]
Put
\[
\omega=\varphi(\xi).
\]
Then
\[
\omega
=
\tau^{-1}(\zeta)
=
\delta\circ\widetilde\psi\circ\delta(\zeta)
=
\frac{1}{\overline{\psi(\zeta)}}.
\]
Since
\[
|\omega|\leqslant1
\qquad\text{and}\qquad
|\psi(\zeta)|\leqslant1,
\]
the last identity implies
\[
|\omega|=|\psi(\zeta)|=1.
\]
Hence
\[
\omega=\psi(\zeta)\in\mathbb T,
\]
to get  that
\[
\varphi(\xi)
=
\psi(\zeta)
\in\mathbb T.
\]
Thus
\[
\varphi(\mathbb T)
\cap
\psi(\mathbb T)
\cap
\mathbb T
\neq\varnothing.
\]

Conversely, suppose that there exist
$\xi,\zeta,\omega_0\in\mathbb T$ such that
\[
\varphi(\xi)
=
\psi(\zeta)
=
\omega_0.
\]
By \textup{(3.3)} and $|\omega_0|=1$,
\[
\tau^{-1}(\zeta)
=
\frac{1}{\overline{\psi(\zeta)}}
=
\omega_0.
\]
Thus
\[
\tau(\omega_0)=\zeta,
\]
and hence
\[
(\tau\circ\varphi)(\xi)=\zeta\in\mathbb T.
\]
It follows that
\[
\overline{(\tau\circ\varphi)(\mathbb D)}
\not\subset\mathbb D.
\]
The proof is complete.
\end{proof}

We can now state the main result of this section.

\begin{theorem}\label{thm3.4}
Let $\alpha>-1$ and $0<p<\infty$. Suppose that $\varphi$ and $\psi$
are nonconstant linear fractional self-maps of $\mathbb D$. Then
\[
C_\varphi C_\psi^*
\in S_p(\mathcal D_\alpha)
\]
if and only if
\[
\varphi(\mathbb T)
\cap
\psi(\mathbb T)
\cap
\mathbb T
=
\varnothing.
\]
Moreover, whenever this condition holds,
\[
C_\varphi C_\psi^*\in S_q(\mathcal D_\alpha)
\]
for every $q\in(0,\infty)$.
\end{theorem}

\begin{proof}
Let $\tau$ be the Krein adjoint of $\psi$. By
Lemma~\ref{lem3.1},
\[
C_\varphi C_\psi^*
\in S_p(\mathcal D_\alpha)
\quad\Longleftrightarrow\quad
C_{\tau\circ\varphi}
\in S_p(\mathcal D_\alpha).
\]
By Proposition~\ref{prop3.2},
\[
C_{\tau\circ\varphi}
\in S_p(\mathcal D_\alpha)
\quad\Longleftrightarrow\quad
\overline{(\tau\circ\varphi)(\mathbb D)}
\subset\mathbb D.
\]
Finally, Lemma~\ref{lem3.3} gives
\[
\overline{(\tau\circ\varphi)(\mathbb D)}
\subset\mathbb D
\quad\Longleftrightarrow\quad
\varphi(\mathbb T)
\cap
\psi(\mathbb T)
\cap
\mathbb T
=
\varnothing.
\]
From the three equivalences above, membership in one Schatten class is equivalent
to
\[
\varphi(\mathbb T)
\cap
\psi(\mathbb T)
\cap
\mathbb T
=
\varnothing,
\]
and this condition does not depend on $p$. This completes the proof.
\end{proof}

\begin{remark}
The above reduction is specific to the product $C_\varphi C_\psi^*$.
Indeed, for the reverse product, we have
\[
D_\psi^*D_\varphi
=
M_gC_\tau M_v^*M_{\varphi'}C_\varphi.
\]
Here, the adjoint $M_v^*$ prevents the two composition operators
from being combined into a single composition operator. Hence the
preceding argument does not directly apply to $C_\psi^*C_\varphi$.
\end{remark}

\section{The reverse product $C_\psi^*C_\varphi$}

The preceding remark shows that the exact factorization used for
$C_\varphi C_\psi^*$ does not directly apply to the reverse product.
We now give a separate argument.  The necessity of the boundary
condition will be obtained from a compactness reduction modulo the compact
operators, while the sufficiency will follow from finite rank
approximations with exponentially decaying tails.

For a Hilbert space $\mathcal H$, we denote by
$\mathcal B(\mathcal H)$ the algebra of bounded linear operators on
$\mathcal H$ and by $\mathcal K(\mathcal H)$ the ideal of compact
operators. For $A,B\in\mathcal B(\mathcal H)$, their commutator is
defined by
\[
[A,B]=AB-BA.
\]
We also write
\[
\mathcal C(\mathcal H)
=
\mathcal B(\mathcal H)/\mathcal K(\mathcal H)
\]
for the Calkin algebra.

We begin with the adjoint formula on the weighted Dirichlet space.
\begin{lemma}\label{lem4.1}
Let $\alpha>-1$, and let
\[
\eta(z)=\frac{az+b}{cz+d}
\]
be a nonconstant linear fractional self-map of $\mathbb D$.  Let $\sigma$
be the Krein adjoint of $\eta$.  Then
\[
C_\eta^*
-
M_{g_\eta}C_\sigma M_{h_\eta}^*
\]
is compact on $\mathcal D_\alpha$, where
\[
g_\eta(z)=(-\overline b z+\overline d)^{-\alpha}  \ \ \ \ \text{and}\ \ \ \
h_\eta(z)=(cz+d)^\alpha.
\]
The powers are taken with compatible analytic branches on a neighborhood of $\overline{\mathbb D}$, normalized so that $g_\eta(0)\overline{h_\eta(0)}=1$.
\end{lemma}

\begin{proof}
Let
\[
\beta_n=\|z^n\|_{\mathcal D_\alpha},
\qquad n\geqslant1.
\]
A direct calculation gives
\[
\beta_n^2
=
\Gamma(\alpha+2)
\frac{n^2\Gamma(n)}{\Gamma(n+\alpha+1)},
\]
to get that
\[
\beta_n
\asymp
n^{(1-\alpha)/2},\]
since
\[\lim_{n\to\infty}
\frac{\beta_n}{n^{(1-\alpha)/2}}
=
\sqrt{\Gamma(\alpha+2)}.
\]
Thus $\mathcal D_\alpha$ is a weighted Hardy space $H^2(\beta)$ in the
sense of \cite{CuckovicLe2016}, with
\(
t=\frac{1-\alpha}{2}.
\)
Using $2t-1=-\alpha$, we have by \cite[Theorem 3.2]{CuckovicLe2016}  that
\[
C_\eta^*
-
M_{g_\eta}C_\sigma M_{h_\eta}^*
\in\mathcal K(\mathcal D_\alpha).
\]
This completes the proof.
\end{proof}

We shall also use the following elementary commutator fact.

\begin{lemma}\label{lem4.2}
Suppose that $u$ and $v$ are analytic on a neighborhood of
$\overline{\mathbb D}$. Then
\[
[M_u^*,M_v]
\in
\mathcal K(\mathcal D_\alpha).
\]
\end{lemma}

\begin{proof}
Put
\[
\beta_0=1,
\qquad
\beta_n=\|z^n\|_{\mathcal D_\alpha},
\quad n\geqslant1,
\]
and let
\[
e_n=\frac{z^n}{\beta_n},
\qquad n\geqslant0.
\]
Then $\{e_n\}_{n=0}^\infty$ is an orthonormal basis of
$\mathcal D_\alpha$. Moreover,
\[
M_ze_n=\omega_ne_{n+1},
\qquad
\omega_n=\frac{\beta_{n+1}}{\beta_n}.
\]
By the asymptotic formula obtained in the proof of
Lemma~\ref{lem4.1}, we have
\[
\beta_n
\sim
\sqrt{\Gamma(\alpha+2)}\,n^{(1-\alpha)/2},
\]
and hence
\[
\omega_n\rightarrow 1 \ \ \ \ (n\rightarrow \infty).
\]

Since $\{e_n\}_{n=0}^\infty$ is an orthonormal basis, the adjoint of $M_z$
satisfies
\[
M_z^*e_0=0,
\qquad
M_z^*e_{n+1}=\omega_ne_n,
\quad n\geqslant0.
\]
Indeed, for $j,n\geqslant0$,
\[
\langle M_z^*e_{n+1},e_j\rangle
=
\langle e_{n+1},M_ze_j\rangle
=
\omega_j\langle e_{n+1},e_{j+1}\rangle,
\]
which vanishes unless $j=n$.  It follows that
\[
[M_z^*,M_z]e_0=\omega_0^2e_0,
\]
while,  for $n\geqslant1$,
\[
[M_z^*,M_z]e_n
=
\bigl(\omega_n^2-\omega_{n-1}^2\bigr)e_n.
\]
Thus $[M_z^*,M_z]$ is diagonal with respect to the orthonormal basis
$\{e_n\}_{n=0}^\infty$. Since $\omega_n\to1$,
\[
\omega_n^2-\omega_{n-1}^2\rightarrow 0 \ \ \ \ (n\rightarrow \infty).
\]
Hence its diagonal entries tend to zero. To see explicitly that the
operator is compact, let $T_N$ denote the diagonal operator obtained
from $[M_z^*,M_z]$ by retaining only its first $N+1$ diagonal entries.
Then $T_N$ has finite rank and
\[
\|[M_z^*,M_z]-T_N\|
=
\sup_{n>N}
\left|\omega_n^2-\omega_{n-1}^2\right|
\rightarrow 0 \ \ \ \  (N\rightarrow \infty).
\]
Therefore,
\[
[M_z^*,M_z]\in\mathcal K(\mathcal D_\alpha).
\]

Let
\[
\pi:
\mathcal B(\mathcal D_\alpha)
\longrightarrow
\mathcal B(\mathcal D_\alpha)/\mathcal K(\mathcal D_\alpha)
\]
be the quotient map. Since
\[
\pi([M_z^*,M_z])=0,
\]
we have
\[
\pi(M_z)^*\pi(M_z)
=
\pi(M_z)\pi(M_z)^*,
\]
so $\pi(M_z)$ is normal in the Calkin algebra.

Now let $p$ and $q$ be analytic polynomials. Since
\[
M_p=p(M_z),
\qquad
M_q=q(M_z),
\]
and $\pi(M_z)$ commutes with its adjoint, we obtain
\[
\pi(M_p)^*\pi(M_q)
=
\pi(M_q)\pi(M_p)^*.
\]
Thus
\[
\pi([M_p^*,M_q])=0,
\]
and hence
\[
[M_p^*,M_q]
\in
\mathcal K(\mathcal D_\alpha).
\]

Finally, let $u$ and $v$ be analytic on a neighborhood of
$\overline{\mathbb D}$.  There are two sequences of polynomials $\{p_n\}_{n=1}^\infty$ and
$\{q_n\}_{n=1}^\infty$ such that
\[
p_n\rightarrow u,
\qquad
p_n'\rightarrow u',
\]
and
\[
q_n\rightarrow v,
\qquad
q_n'\rightarrow v'
\]
uniformly on $\overline{\mathbb D}$.

For a function $w$ analytic on a neighborhood of
$\overline{\mathbb D}$, we have
\[
(wf)'=w'f+wf'.
\]
Using the continuous embedding
$\mathcal D_\alpha\rightarrow A_\alpha^2$, it follows that
\[
\|M_wf\|_{\mathcal D_\alpha}
\lesssim
\bigl(\|w\|_\infty+\|w'\|_\infty\bigr)
\|f\|_{\mathcal D_\alpha}.
\]
Consequently,
\[
M_{p_n}\rightarrow M_u\ \ \ \ \text{and} \ \ \ \
M_{q_n}\rightarrow M_v
\]
as $n\rightarrow \infty$ in the operator norm. Since taking adjoints preserves the operator norm, we also have
\[
M_{p_n}^*\rightarrow M_u^* \ \ \ \ \mathrm{and} \ \ \ \
M_{q_n}^*\rightarrow M_v^*
\]
as $n\rightarrow \infty$ in the operator norm. Moreover, the sequences
$\{M_{p_n}\}_{n=1}^\infty$ and $\{M_{q_n}\}_{n=1}^\infty$ are uniformly bounded. Hence
\begin{align*}
&\bigl\|
[M_{p_n}^*,M_{q_n}]
-
[M_u^*,M_v]
\bigr\| \\
&\leqslant
\bigl\|
M_{p_n}^*M_{q_n}-M_u^*M_v
\bigr\|
+
\bigl\|
M_{q_n}M_{p_n}^*-M_vM_u^*
\bigr\| \\
&\leqslant
\|M_{p_n}^*-M_u^*\|\,\|M_{q_n}\|
+
\|M_u^*\|\,\|M_{q_n}-M_v\| \\
&\quad+
\|M_{q_n}-M_v\|\,\|M_{p_n}^*\|
+
\|M_v\|\,\|M_{p_n}^*-M_u^*\|.
\end{align*}
Therefore,
\[
[M_{p_n}^*,M_{q_n}]
\rightarrow
[M_u^*,M_v]
\]
 as $n\rightarrow \infty$ in the  operator norm.

For each $n$, the functions $p_n$ and $q_n$ are analytic polynomials,
and hence, by the preceding argument,
\[
[M_{p_n}^*,M_{q_n}]
\in
\mathcal K(\mathcal D_\alpha).
\]
Since $\mathcal K(\mathcal D_\alpha)$ is closed in the operator norm,
the operator-norm limit is also compact. Thus
\[
[M_u^*,M_v]
\in
\mathcal K(\mathcal D_\alpha),
\]
which finishes the proof.
\end{proof}

We next reduce compactness of the reverse product to compactness of a
single composition operator.

\begin{lemma}\label{lem4.3}
Let $\alpha>-1$, and let $\varphi$ and $\psi$ be nonconstant linear
fractional self-maps of $\mathbb D$.  Let $\tau$ be the Krein adjoint of
$\psi$.  Then
\[
C_\psi^*C_\varphi
\in\mathcal K(\mathcal D_\alpha)
\quad\Longleftrightarrow\quad
C_{\varphi\circ\tau}
\in\mathcal K(\mathcal D_\alpha).
\]
\end{lemma}

\begin{proof}
Let $\sigma$ be the Krein adjoint of $\varphi$. By the
adjoint formula on $\mathcal D_\alpha$ and taking adjoints, we have
\[
C_\psi^*
\equiv
M_gC_\tau M_h^*,
\qquad
C_\varphi
\equiv
M_HC_\sigma^*M_G^*
\pmod{\mathcal K},
\]
where the multiplier symbols are those in
Lemma~\ref{lem4.1}. Hence
\[
C_\psi^*C_\varphi
\equiv
M_gC_\tau M_h^*M_HC_\sigma^*M_G^*
\pmod{\mathcal K}.
\]
By Lemma~\ref{lem4.2},
\[
M_h^*M_H\equiv M_HM_h^*
\pmod{\mathcal K},
\]
and therefore
\[
\begin{aligned}
C_\psi^*C_\varphi
&\equiv
M_gC_\tau M_HM_h^*C_\sigma^*M_G^*
\\
&=
M_{g(H\circ\tau)}
C_\tau C_\sigma^*
M_{h\circ\sigma}^*M_G^*
\pmod{\mathcal K}.
\end{aligned}
\]
Since all multiplier factors are bounded and invertible, it follows that
\[
C_\psi^*C_\varphi\in\mathcal K(\mathcal D_\alpha)
\quad\Longleftrightarrow\quad
C_\tau C_\sigma^*
\in\mathcal K(\mathcal D_\alpha).
\]

It remains to reduce the latter compactness question to a single composition operator. By the finite rank block decomposition and the unitary equivalence induced by the differentiation map $U_\alpha f=f'$ in Proposition~\ref{prop2.1}, we have
\[
C_\tau C_\sigma^*\in\mathcal K(\mathcal D_\alpha)
\quad\Longleftrightarrow\quad
D_\tau D_\sigma^*\in\mathcal K(A_\alpha^2).
\]

We now derive a factorization for $D_\tau D_\sigma^*$ on
$A_\alpha^2$. Since
\[
D_\sigma=M_{\sigma'}C_\sigma,
\]
we have
\[
D_\sigma^*
=
C_\sigma^*M_{\sigma'}^*.
\]
Applying Theorem \ref{thm2.4} to the linear fractional
self-map $\sigma$ gives
\[
C_\sigma^*
=
M_{\widetilde g}
C_{\widehat\sigma}
M_{\widetilde h}^*,
\]
where $\widehat\sigma$ denotes the Krein adjoint of $\sigma$.
Therefore,
\[
D_\sigma^*
=
M_{\widetilde g}
C_{\widehat\sigma}
M_{\widetilde h}^*M_{\sigma'}^*
=
M_{\widetilde g}
C_{\widehat\sigma}
M_{\widetilde v}^*,
\]
where
\[
\widetilde v=\sigma'\widetilde h .
\]

Since
\[
D_\tau=M_{\tau'}C_\tau,
\]
we obtain
\[
\begin{aligned}
D_\tau D_\sigma^*
&=
M_{\tau'}C_\tau
M_{\widetilde g}
C_{\widehat\sigma}
M_{\widetilde v}^*
=
M_{\tau'(\widetilde g\circ\tau)}
C_{\widehat\sigma\circ\tau}
M_{\widetilde v}^*,
\end{aligned}
\]
where we used the identity
\[
C_\tau M_{\widetilde g}
=
M_{\widetilde g\circ\tau}C_\tau .
\]

On the other hand, we have
\[
D_{\widehat\sigma\circ\tau}
=
M_{(\widehat\sigma'\circ\tau)\tau'}
C_{\widehat\sigma\circ\tau}.
\]
Define
\[
\widetilde u
=
\frac{\widetilde g\circ\tau}
{\widehat\sigma'\circ\tau}.
\]
Then we obtain that
\[
D_\tau D_\sigma^*
=
M_{\widetilde u}
D_{\widehat\sigma\circ\tau}
M_{\widetilde v}^* .
\]
The multiplication operators
$M_{\widetilde u}$ and $M_{\widetilde v}$ are bounded and invertible
on $A_\alpha^2$. Therefore,
\[
D_\tau D_\sigma^*
\in\mathcal K(A_\alpha^2)
\quad\Longleftrightarrow\quad
D_{\widehat\sigma\circ\tau}
\in\mathcal K(A_\alpha^2).
\]

Finally, the Krein adjoint operation is involutive. Since $\sigma$ is
the Krein adjoint of $\varphi$, we have
\[
\widehat\sigma=\varphi .
\]
Consequently,
\[
D_\tau D_\sigma^*
\in\mathcal K(A_\alpha^2)
\quad\Longleftrightarrow\quad
D_{\varphi\circ\tau}
\in\mathcal K(A_\alpha^2).
\]
Applying Proposition~\ref{prop2.1} once more yields
\[
D_{\varphi\circ\tau}
\in\mathcal K(A_\alpha^2)
\quad\Longleftrightarrow\quad
C_{\varphi\circ\tau}
\in\mathcal K(\mathcal D_\alpha).
\]
Combining the above equivalences proves the lemma.
\end{proof}

For later use,  we record the compactness criterion for a single linear
fractional symbol.

\begin{lemma}\label{lem4.4}
Let $\alpha>-1$, and let $\eta$ be a nonconstant linear fractional
self-map of $\mathbb D$. Then
\[
C_\eta\in\mathcal K(\mathcal D_\alpha)
\quad\Longleftrightarrow\quad
\overline{\eta(\mathbb D)}\subset\mathbb D.
\]
\end{lemma}

\begin{proof}
By the finite rank block decomposition and
Proposition~\ref{prop2.1},
\[
C_\eta\in\mathcal K(\mathcal D_\alpha)
\quad\Longleftrightarrow\quad
D_\eta\in\mathcal K(A_\alpha^2),
\]
where
\[
D_\eta=M_{\eta'}C_\eta.
\]
As noted above, $\eta'$ and $1/\eta'$ are bounded analytic functions
on a neighborhood of $\overline{\mathbb D}$. Hence $M_{\eta'}$ is
bounded and invertible on $A_\alpha^2$, and therefore
\[
D_\eta\in\mathcal K(A_\alpha^2)
\quad\Longleftrightarrow\quad
C_\eta\in\mathcal K(A_\alpha^2).
\]
Consequently,
\[
C_\eta\in\mathcal K(\mathcal D_\alpha)
\quad\Longleftrightarrow\quad
C_\eta\in\mathcal K(A_\alpha^2).
\]

By
\cite[Theorem 3.5]{MacCluerShapiro1986}, $C_\eta$ is compact on $A_\alpha^2$ if
and only if $\eta$ has no finite angular derivative at any point of
$\mathbb T$.

For a nonconstant linear fractional self-map, this is equivalent to
\[
\overline{\eta(\mathbb D)}\subset\mathbb D.
\]
Indeed, if $\overline{\eta(\mathbb D)}$ meets $\mathbb T$, then there
exists $\zeta\in\mathbb T$ such that $\eta(\zeta)\in\mathbb T$.
Since $\eta$ extends analytically across $\mathbb T$ and
$\eta'(\zeta)\neq0$, it has a finite angular derivative at $\zeta$.
Conversely, if $\eta$ has a finite angular derivative at some
$\zeta\in\mathbb T$, then its angular limit belongs to $\mathbb T$,
and hence $\overline{\eta(\mathbb D)}$ meets $\mathbb T$. Therefore,
\[
C_\eta\in\mathcal K(\mathcal D_\alpha)
\quad\Longleftrightarrow\quad
\overline{\eta(\mathbb D)}\subset\mathbb D,
\]
to finish the proof.
\end{proof}

The geometric condition for the reverse product is expressed in terms of
common boundary contact in the domain.

\begin{lemma}\label{lem4.5}
Let $\varphi$ and $\psi$ be nonconstant linear fractional self-maps of
$\mathbb D$, and let $\tau$ be the Krein adjoint of $\psi$.  Then
\[
\overline{(\varphi\circ\tau)(\mathbb D)}
\subset\mathbb D
\]
if and only if
\[
\varphi^{-1}(\mathbb T)
\cap
\psi^{-1}(\mathbb T)
\cap
\mathbb T
=
\varnothing.
\]
\end{lemma}

\begin{proof}
By \textup{(3.3)},
\[
\tau
=
\delta\circ\widetilde\psi^{-1}\circ\delta.
\]
Hence, we have
\[
\begin{aligned}
(\varphi\circ\tau)(\mathbb D)
=
(\varphi\circ\delta\circ\widetilde\psi^{-1}\circ\delta)(\mathbb D)
=
(\varphi\circ\delta\circ\widetilde\psi^{-1})
\bigl(\widehat{\mathbb C}\backslash\overline{\mathbb D}\bigr)
\subset
(\varphi\circ\delta)
\bigl(\widehat{\mathbb C}\backslash\overline{\mathbb D}\bigr)
=
\varphi(\mathbb D)
\subset
\mathbb D.
\end{aligned}
\]

Assume first that
\[
\overline{(\varphi\circ\tau)(\mathbb D)}
\not\subset\mathbb D.
\]
Since $\varphi\circ\tau$ is a linear fractional self-map and extends
continuously to $\overline{\mathbb D}$, there exist
$\xi,\zeta\in\mathbb T$ such that
\[
(\varphi\circ\tau)(\xi)=\zeta.
\]
Put
\[
\lambda=\tau(\xi).
\]
Then
\[
\varphi(\lambda)=\zeta\in\mathbb T.
\]
Since $\lambda\in\overline{\mathbb D}$ and
$\varphi(\mathbb D)\subset\mathbb D$, necessarily
\[
\lambda\in\mathbb T.
\]
Moreover, by \textup{(3.3)},
\[
\xi
=
\tau^{-1}(\lambda)
=
\frac{1}{\overline{\psi(\lambda)}}.
\]
Since
\[
|\xi|=1
\qquad\text{and}\qquad
|\psi(\lambda)|\leqslant1,
\]
we obtain
\[
|\psi(\lambda)|=1.
\]
Hence
\[
\lambda
\in
\varphi^{-1}(\mathbb T)
\cap
\psi^{-1}(\mathbb T)
\cap
\mathbb T.
\]
Thus
\[
\varphi^{-1}(\mathbb T)
\cap
\psi^{-1}(\mathbb T)
\cap
\mathbb T
\neq\varnothing.
\]

Conversely, suppose that there exists $\lambda\in\mathbb T$ such that
\[
\varphi(\lambda)\in\mathbb T,
\qquad
\psi(\lambda)\in\mathbb T.
\]
Put
\[
\xi=\psi(\lambda)\in\mathbb T.
\]
Since $|\psi(\lambda)|=1$, \textup{(3.3)} gives
\[
\tau^{-1}(\lambda)
=
\frac{1}{\overline{\psi(\lambda)}}
=
\psi(\lambda)
=
\xi.
\]
Therefore
\[
\tau(\xi)=\lambda,
\]
and hence
\[
(\varphi\circ\tau)(\xi)
=
\varphi(\lambda)
\in\mathbb T.
\]
It follows that
\[
\overline{(\varphi\circ\tau)(\mathbb D)}
\not\subset\mathbb D.
\]
The proof is complete.
\end{proof}

Combining the preceding three lemmas gives the compactness characterization.

\begin{proposition}\label{prop4.6}
Let $\alpha>-1$, and let $\varphi$ and $\psi$ be nonconstant linear
fractional self-maps of $\mathbb D$.  Then
\[
C_\psi^*C_\varphi
\in\mathcal K(\mathcal D_\alpha)
\]
if and only if
\[
\varphi^{-1}(\mathbb T)
\cap
\psi^{-1}(\mathbb T)
\cap
\mathbb T
=
\varnothing.
\]
\end{proposition}

\begin{proof}
By Lemma~\ref{lem4.3}, we have
\[
C_\psi^*C_\varphi
\in\mathcal K(\mathcal D_\alpha)
\quad\Longleftrightarrow\quad
C_{\varphi\circ\tau}
\in\mathcal K(\mathcal D_\alpha).
\]
Then using Lemmas \ref{lem4.4} and
\ref{lem4.5}, we  obtain the desired conclusion.
\end{proof}

To show that the above boundary condition also yields Schatten $p$-class membership, we need the finite rank approximations below.

\begin{lemma}\label{lem4.7}
Let $\alpha>-1$, and let $\varphi$ and $\psi$ be nonconstant linear
fractional self-maps of $\mathbb D$.  Assume that
\[
\varphi^{-1}(\mathbb T)
\cap
\psi^{-1}(\mathbb T)
\cap
\mathbb T
=
\varnothing.
\]
Let
\[
T=C_\psi^*C_\varphi.
\]
Then there exist constants $C>0$ and $0<\rho<1$ and finite rank
operators $F_N$ such that
\[
\operatorname{rank}(F_N)\leqslant2(N+1)
\]
and
\[
\|T-F_N\|\leqslant C\rho^N,
\qquad N\geqslant0.
\]
\end{lemma}

\begin{proof}
Let
\[
\beta_0=1,
\qquad
\beta_n=\|z^n\|_{\mathcal D_\alpha}
\quad (n\geqslant1),
\]
and set
\[
e_n=\frac{z^n}{\beta_n}.
\]
Then $\{e_n\}_{n=0}^\infty$ is an orthonormal basis of
$\mathcal D_\alpha$, and
\[
\beta_n^2\asymp n^{1-\alpha}.
\tag{4.2}
\]
Let $P_N$ be the orthogonal projection onto
\[
\operatorname{span}\{e_0,\ldots,e_N\},
\]
and put
\[
Q_N=I-P_N.
\]

We first use the boundary separation hypothesis to choose the constants
that will appear in the tail estimate. Since $\varphi$ and $\psi$
extend continuously to $\overline{\mathbb D}$, the function
\[
z\longmapsto
\min\{|\varphi(z)|,|\psi(z)|\}
\]
is continuous on $\overline{\mathbb D}$. Hence
\[
r_0
:=
\max_{z\in\overline{\mathbb D}}
\min\{|\varphi(z)|,|\psi(z)|\}
\]
is well defined. We claim that
\[
r_0<1.
\tag{4.3}
\]
Indeed, if $r_0=1$, then there exists
$\zeta\in\overline{\mathbb D}$ such that
\[
|\varphi(\zeta)|=|\psi(\zeta)|=1.
\]
Since $\varphi$ and $\psi$ are nonconstant self-maps of
$\mathbb D$, the maximum modulus principle implies that
$\zeta\notin\mathbb D$. Thus $\zeta\in\mathbb T$, and hence
\[
\zeta
\in
\varphi^{-1}(\mathbb T)
\cap
\psi^{-1}(\mathbb T)
\cap
\mathbb T,
\]
contrary to the hypothesis.

Choose and fix numbers $r$ and $\rho$ such that
\[
\max\{r_0,|\varphi(0)|,|\psi(0)|\}
<r<\rho<1.
\]
We next establish a tail estimate corresponding to these fixed
numbers. If
\[
f=\sum_{n=0}^{\infty}a_ne_n,
\]
then
\[
Q_Nf=\sum_{n>N}a_ne_n,
\]
and, since $\{e_n\}_{n=0}^\infty$ is an orthonormal basis,
\[
\sum_{n>N}|a_n|^2
=
\|Q_Nf\|_{\mathcal D_\alpha}^2
\leqslant
\|f\|_{\mathcal D_\alpha}^2.
\]
For $|w|\leqslant r$, Cauchy--Schwarz and \textup{(4.2)} give
\[
\begin{aligned}
|Q_Nf(w)|^2
&=
\left|
\sum_{n>N}a_n\frac{w^n}{\beta_n}
\right|^2\\
&\leqslant
\left(\sum_{n>N}|a_n|^2\right)
\left(\sum_{n>N}\frac{|w|^{2n}}{\beta_n^2}\right)\\
&\lesssim
\|f\|_{\mathcal D_\alpha}^2
\sum_{n>N}n^{\alpha-1}r^{2n}.
\end{aligned}
\]
Since $r<\rho$, let
$
q=\big(\frac r\rho\big)^2<1.
$
Then
\[
\begin{aligned}
\sum_{n>N}n^{\alpha-1}r^{2n}
=
\sum_{n>N}
n^{\alpha-1}q^n\rho^{2n}
\leqslant
\rho^{2N}
\sum_{n=1}^{\infty}n^{\alpha-1}q^n
\lesssim
\rho^{2N},
\end{aligned}
\]
because the last series converges. Therefore,
\[
|Q_Nf(w)|
\lesssim
\rho^N\|f\|_{\mathcal D_\alpha}.
\]

Similarly,
\[
(Q_Nf)'(w)
=
\sum_{n>N}
a_n\frac{nw^{n-1}}{\beta_n},
\]
and hence, again by the  Cauchy--Schwarz inequality  and \textup{(4.2)},
\[
\begin{aligned}
|(Q_Nf)'(w)|^2
&\leqslant
\left(\sum_{n>N}|a_n|^2\right)
\left(
\sum_{n>N}
\frac{n^2|w|^{2n-2}}{\beta_n^2}
\right)\\
&\lesssim
\|f\|_{\mathcal D_\alpha}^2
\sum_{n>N}
n^{\alpha+1}r^{2n-2}.
\end{aligned}
\]
Thus,
\[
\begin{aligned}
\sum_{n>N}
n^{\alpha+1}r^{2n-2}
&=
\frac1{r^2}
\sum_{n>N}
n^{\alpha+1}q^n\rho^{2n}\\
&\leqslant
\frac{\rho^{2N}}{r^2}
\sum_{n=1}^{\infty}
n^{\alpha+1}q^n\\
&\lesssim
\rho^{2N}.
\end{aligned}
\]
Consequently,
\[
\sup_{|w|\leqslant r}|Q_Nf(w)|
\lesssim
\rho^N\|f\|_{\mathcal D_\alpha}
\quad\text{and}\quad
\sup_{|w|\leqslant r}|(Q_Nf)'(w)|
\lesssim
\rho^N\|f\|_{\mathcal D_\alpha}.
\tag{4.4}
\]

Now set
\[
A=\{z\in\mathbb D:|\varphi(z)|\leqslant r\} \ \ \ \ \mathrm{and}\ \ \ \
B=\mathbb D\backslash A.
\]
If $z\in B$, then
\[
|\varphi(z)|>r>r_0.
\]
By the definition of $r_0$, we have
\[
|\psi(z)|\leqslant r_0<r,
\qquad z\in B.
\tag{4.5}
\]

Let $\|f\|_{\mathcal D_\alpha}\leqslant 1$ and
$\|g\|_{\mathcal D_\alpha}\leqslant1$. Since $Q_N$ is an orthogonal
projection and $T=C_\psi^*C_\varphi$,
\[
\begin{aligned}
\langle Q_NTQ_Nf,g\rangle_{\mathcal D_\alpha}
&=
\langle TQ_Nf,Q_Ng\rangle_{\mathcal D_\alpha}=
\langle C_\varphi Q_Nf,
C_\psi Q_Ng\rangle_{\mathcal D_\alpha}.
\end{aligned}
\]
The contribution of the value at the origin is
\[
Q_Nf(\varphi(0))
\overline{Q_Ng(\psi(0))}.
\]
Since
\[
|\varphi(0)|<r \ \ \ \  \mathrm{and}\ \ \ \
|\psi(0)|<r,
\]
the estimate \textup{(4.4)} yields
\[
\left|
Q_Nf(\varphi(0))
\overline{Q_Ng(\psi(0))}
\right|
\lesssim
\rho^{2N}.
\]

It remains to estimate the derivative part. Split the integral over
$A$ and $B$. On $A$, we have $|\varphi(z)|\leqslant r$, so
\textup{(4.4)} gives
\[
|(Q_Nf)'(\varphi(z))|
\lesssim
\rho^N.
\]
Hence, by the Cauchy--Schwarz inequality,
\[
\begin{aligned}
&\int_A
|(Q_Nf)'(\varphi(z))\varphi'(z)|
|(Q_Ng)'(\psi(z))\psi'(z)|
\,dA_\alpha(z)
\\
&\quad\lesssim
\rho^N
\int_A
|\varphi'(z)|
|(Q_Ng)'(\psi(z))\psi'(z)|
\,dA_\alpha(z)
\\
&\quad\leqslant
\rho^N
\left(
\int_{\mathbb D}|\varphi'(z)|^2\,dA_\alpha(z)
\right)^{1/2}
\left(
\int_{\mathbb D}
|(Q_Ng)'(\psi(z))|^2|\psi'(z)|^2
\,dA_\alpha(z)
\right)^{1/2}
\\
&\quad=
\rho^N
\|\varphi'\|_{A_\alpha^2}
\,
\|D_\psi((Q_Ng)')\|_{A_\alpha^2}
\\
&\quad\lesssim
\rho^N.
\end{aligned}
\]
Here we used the boundedness of $D_\psi$ on $A_\alpha^2$ and
\[
\|(Q_Ng)'\|_{A_\alpha^2}
=
\|Q_Ng\|_{\mathcal D_\alpha}
\leqslant1.
\]

On $B$, \textup{(4.5)} gives $|\psi(z)|<r$, and therefore
\textup{(4.4)} yields
\[
|(Q_Ng)'(\psi(z))|
\lesssim
\rho^N.
\]
It follows that
\[
\begin{aligned}
&\int_B
|(Q_Nf)'(\varphi(z))\varphi'(z)|
|(Q_Ng)'(\psi(z))\psi'(z)|
\,dA_\alpha(z)
\\
&\quad\lesssim
\rho^N
\int_B
|(Q_Nf)'(\varphi(z))\varphi'(z)|
|\psi'(z)|
\,dA_\alpha(z)
\\
&\quad\leqslant
\rho^N
\left(
\int_{\mathbb D}
|(Q_Nf)'(\varphi(z))|^2|\varphi'(z)|^2
\,dA_\alpha(z)
\right)^{1/2}
\left(
\int_{\mathbb D}|\psi'(z)|^2\,dA_\alpha(z)
\right)^{1/2}
\\
&\quad=
\rho^N
\|D_\varphi((Q_Nf)')\|_{A_\alpha^2}
\,
\|\psi'\|_{A_\alpha^2}
\\
&\quad\lesssim
\rho^N,
\end{aligned}
\]
where we used the boundedness of $D_\varphi$.

Combining the estimate of the value at the origin with the estimates
over $A$ and $B$, we obtain
\[
\left|
\langle Q_NTQ_Nf,g\rangle_{\mathcal D_\alpha}
\right|
\lesssim
\rho^N.
\]
Taking the supremum over all
$\|f\|_{\mathcal D_\alpha}\leqslant 1$ and $
\|g\|_{\mathcal D_\alpha}\leqslant1$ gives
\[
\|Q_NTQ_N\|
\lesssim
\rho^N.
\tag{4.6}
\]

Finally, define
\[
F_N
=
TP_N+P_NT-P_NTP_N.
\]
Since $Q_N=I-P_N$, we have
\[
F_N
=
TP_N+P_NT(I-P_N)
=
TP_N+P_NTQ_N.
\]
Hence
\[
\operatorname{rank}(F_N)
\leqslant
\operatorname{rank}(TP_N)
+
\operatorname{rank}(P_NTQ_N)
\leqslant
2(N+1).
\]
Moreover,
\[
\begin{aligned}
T-F_N
&=
T-TP_N-P_NT+P_NTP_N\\
&=
(I-P_N)T(I-P_N)\\
&=
Q_NTQ_N.
\end{aligned}
\]
Therefore, by \textup{(4.6)},
\[
\|T-F_N\|
=
\|Q_NTQ_N\|
\leqslant
C\rho^N
\]
for some constant $C>0$ independent of $N$. This completes the proof.
\end{proof}

We can now characterize the Schatten class membership of the reverse
product $C_\psi^*C_\varphi$.

\begin{theorem}\label{thm4.8}
Let $\alpha>-1$ and $0<p<\infty$.  Suppose that $\varphi$ and $\psi$
are nonconstant linear fractional self-maps of $\mathbb D$.  Then
\[
C_\psi^*C_\varphi
\in S_p(\mathcal D_\alpha)
\]
if and only if
\[
\varphi^{-1}(\mathbb T)
\cap
\psi^{-1}(\mathbb T)
\cap
\mathbb T
=
\varnothing.
\]
If this condition holds, then
$$
C_\psi^*C_\varphi
\in S_q(\mathcal D_\alpha)$$
for every $0<q<\infty$.
\end{theorem}

\begin{proof}
Suppose first that
\[
C_\psi^*C_\varphi
\in S_p(\mathcal D_\alpha).
\]
Every Schatten class operator is compact, so
Proposition~\ref{prop4.6} gives
\[
\varphi^{-1}(\mathbb T)
\cap
\psi^{-1}(\mathbb T)
\cap
\mathbb T
=
\varnothing.
\]

Conversely, assume that
\[
\varphi^{-1}(\mathbb T)
\cap
\psi^{-1}(\mathbb T)
\cap
\mathbb T
=
\varnothing.
\]
By Lemma~\ref{lem4.7}, there exist finite rank
operators $F_N$, a constant $C>0$, and $0<\rho<1$ such that
\[
\operatorname{rank}(F_N)\leqslant2(N+1)
\]
and
\[
\|C_\psi^*C_\varphi-F_N\|
\leqslant
C\rho^N,
\qquad N\geqslant0.
\]
Since
\[
\operatorname{rank}(F_N)\leqslant2N+2<2N+3,
\]
the definition of the singular values \cite{PutinarTener2018} gives
\[
s_{2N+3}(C_\psi^*C_\varphi)
\leqslant
\|C_\psi^*C_\varphi-F_N\|
\leqslant
C\rho^N.
\]

Set
\(
\rho_1=\sqrt{\rho},
\)
so that $0<\rho_1<1$. If $n\geqslant3$ is odd, write
\(
n=2N+3
\)
for some $N\geqslant0$. Then
\[
s_n(C_\psi^*C_\varphi)
=
s_{2N+3}(C_\psi^*C_\varphi)
\leqslant
C\rho^N
=
C\rho^{(n-3)/2}
=
C\rho^{-3/2}\rho_1^n.
\]

If $n\geqslant4$ is even, write
\(
n=2N+4
\)
for some $N\geqslant0$. Since the singular values are nonincreasing,
\[
s_n(C_\psi^*C_\varphi)
=
s_{2N+4}(C_\psi^*C_\varphi)
\leqslant
s_{2N+3}(C_\psi^*C_\varphi)
\leqslant
C\rho^N
=
C\rho^{(n-4)/2}
=
C\rho^{-2}\rho_1^n.
\]

Therefore, after enlarging the constant to include the cases $n=1,2$,
there exists $C_1>0$ such that
\[
s_n(C_\psi^*C_\varphi)\leqslant C_1\rho_1^n,
\qquad n\geqslant1.
\]
Consequently, for every $q>0$,
\[
\sum_{n=1}^{\infty}s_n(C_\psi^*C_\varphi)^q
\leqslant
C_1^q\sum_{n=1}^{\infty}\rho_1^{qn}
<\infty,
\]
since $0<\rho_1^q<1$. Hence, we have
$$
C_\psi^*C_\varphi\in S_q(\mathcal D_\alpha)$$
for every $q>0$.
In particular,
\(
C_\psi^*C_\varphi\in S_p(\mathcal D_\alpha).
\)
This completes the proof.
\end{proof}

\section{Essential norms for composition operator products}

The preceding sections characterize Schatten class membership of the two products in terms of the boundary geometry of their linear fractional symbols. We now study the essential norms of these products, which measure their distance from the compact operators. Adjoint formulas and boundary localization yield exact formulas when $\alpha\geqslant0$, while the derivative model and positive Toeplitz operator estimates provide two-sided bounds when $-1<\alpha<0$.

For a bounded operator $T$ on a Hilbert space $\mathcal H$, its essential
norm is defined by
$$\|T\|_e=\inf\{\|T-K\|:K\in\mathcal K(\mathcal H)\}.$$
Equivalently, if
$\pi:\mathcal B(\mathcal H)\to
\mathcal B(\mathcal H)/\mathcal K(\mathcal H)$
denotes the quotient map from $\mathcal B(\mathcal H)$ onto the Calkin algebra, then
$\|T\|_e=\|\pi(T)\|$.

We first consider the classical Dirichlet space
$\mathcal D=\mathcal D_0$.
In this case, the Cowen-type adjoint formula modulo compact operators
takes a particularly simple form and leads to an exact formula for the
essential norm of the product $C_\varphi C_\psi^*$.

\begin{lemma}\label{lem5.1}
Let $\eta$ be a univalent analytic self-map of $\mathbb D$. If $\eta$
has an angular derivative at some point $\xi\in\mathbb T$, then
\[
\|C_\eta\|_{e,\mathcal D}=1.
\]
\end{lemma}

\begin{proof}
The proof can be found in \cite[Theorem~1.2]{LiLuYu2018}.
\end{proof}

\begin{theorem}\label{thm5.2}
Let $\varphi$ and $\psi$ be nonconstant linear fractional self-maps of
$\mathbb D$. Then
\[
\|C_\varphi C_\psi^*\|_{e,\mathcal D}
=
\begin{cases}
0,
&
\varphi(\mathbb T)\cap\psi(\mathbb T)\cap\mathbb T
=
\varnothing,
\\[2mm]
1,
&
\varphi(\mathbb T)\cap\psi(\mathbb T)\cap\mathbb T
\neq
\varnothing.
\end{cases}
\]
\end{theorem}

\begin{proof}
Let $\tau$ be the Krein adjoint of $\psi$. By Lemma~\ref{lem4.1},
when $\alpha=0$ we have $g_\psi=h_\psi=1$, and hence
\[
C_\psi^*\equiv C_\tau
\pmod{\mathcal K(\mathcal D)}.
\]
It follows that
\[
C_\varphi C_\psi^*
\equiv
C_\varphi C_\tau
=
C_{\tau\circ\varphi}
\pmod{\mathcal K(\mathcal D)}.
\]
Therefore,
\[
\|C_\varphi C_\psi^*\|_{e,\mathcal D}
=
\|C_{\tau\circ\varphi}\|_{e,\mathcal D}.
\tag{5.1}
\]

Set $\eta=\tau\circ\varphi$. By Lemma~\ref{lem3.3},
\[
\overline{\eta(\mathbb D)}\subset\mathbb D
\quad\Longleftrightarrow\quad
\varphi(\mathbb T)\cap\psi(\mathbb T)\cap\mathbb T
=
\varnothing.
\tag{5.2}
\]
Suppose first that
$\varphi(\mathbb T)\cap\psi(\mathbb T)\cap\mathbb T=\varnothing$.
Then \textup{(5.2)} gives
$\overline{\eta(\mathbb D)}\subset\mathbb D$, and hence
$C_\eta$ is compact on $\mathcal D$ by Lemma~\ref{lem4.4}.
Thus \textup{(5.1)} yields
$\|C_\varphi C_\psi^*\|_{e,\mathcal D}=0$.

Now we consider the case that
$\varphi(\mathbb T)\cap\psi(\mathbb T)\cap\mathbb T\neq\varnothing$.
Then \textup{(5.2)} gives
$\overline{\eta(\mathbb D)}\not\subset\mathbb D$.
Since $\eta$ is a nonconstant linear fractional self-map of $\mathbb D$,
there exists $\xi\in\mathbb T$ such that $\eta(\xi)\in\mathbb T$.
Moreover, $\eta$ is univalent and extends
continuously to $\overline{\mathbb D}$, so it has an angular derivative at $\xi$.
By Lemma~\ref{lem5.1},
$\|C_\eta\|_{e,\mathcal D}=1$.
Together with \textup{(5.1)}, this gives
$\|C_\varphi C_\psi^*\|_{e,\mathcal D}=1$.

This completes the proof.
\end{proof}

\begin{theorem}\label{thm5.3}
Let $\varphi$ and $\psi$ be nonconstant linear fractional self-maps of
$\mathbb D$. Then
\[
\|C_\psi^*C_\varphi\|_{e,\mathcal D}
=
\begin{cases}
0,
&
\varphi^{-1}(\mathbb T)\cap\psi^{-1}(\mathbb T)\cap\mathbb T
=
\varnothing,
\\[2mm]
1,
&
\varphi^{-1}(\mathbb T)\cap\psi^{-1}(\mathbb T)\cap\mathbb T
\neq
\varnothing.
\end{cases}
\]
\end{theorem}

\begin{proof}
The proof is similar to that of Theorem~\ref{thm5.2}. The result follows
from Lemmas~\ref{lem4.1}, \ref{lem4.4}, \ref{lem4.5} and \ref{lem5.1}.
\end{proof}

We next turn to the case $\alpha>0$.  We first recall the exact
essential-norm formula of Li, Lu, and Yu for a single composition
operator.

\begin{lemma}\label{lem5.4}
Let $\alpha>0$, and let $\eta$ be a nonconstant linear fractional
self-map of $\mathbb D$. Then
\[
\|C_\eta\|_{e,\mathcal D_\alpha}^2
=
\max_{\substack{\xi\in\mathbb T\\ |\eta(\xi)|=1}}
|\eta'(\xi)|^{-\alpha},
\]
where the right-hand side is understood to be zero if
$\overline{\eta(\mathbb D)}\subset\mathbb D$.
\end{lemma}

\begin{proof}
Since $\eta$ is univalent and analytic on a neighborhood of
$\overline{\mathbb D}$, the result follows directly from
\cite[Theorem~1.1]{LiLuYu2018}.
\end{proof}

\begin{lemma}\label{lem5.5}
Let $\alpha>-1$ and let $s\in C(\overline{\mathbb D})$. Then
\[
\|T_s\|_{e,A_\alpha^2}
=
\max_{\zeta\in\mathbb T}|s(\zeta)|.
\]
\end{lemma}

\begin{proof}
We first prove the lower estimate. Let
$Q\in\mathcal K(A_\alpha^2)$, and let
\[
k_z^\alpha
=
\frac{K_z^\alpha}{\|K_z^\alpha\|_{A_\alpha^2}}
\]
be the normalized reproducing kernel at $z\in\mathbb D$.
Since $k_z^\alpha$ converges weakly to $0$ on $A_\alpha^2$ as $|z|\to1^-$,
we have
\[
\|Qk_z^\alpha\|_{A_\alpha^2}\rightarrow0
\qquad (|z|\to1^-).
\]
For any $\zeta\in\mathbb T$, letting $z\to\zeta$, we obtain
\[
\begin{aligned}
\|T_s-Q\|
&\geqslant
\left|
\left\langle
(T_s-Q)k_z^\alpha,k_z^\alpha
\right\rangle
\right|
=
\left|
B_\alpha s(z)
-
\langle Qk_z^\alpha,k_z^\alpha\rangle
\right|.
\end{aligned}
\]
Since $s\in C(\overline{\mathbb D})$, the boundary behavior of the
Berezin transform gives
$B_\alpha s(z)\to s(\zeta)$ as $z\to\zeta$; see
\cite[Proposition~6.14]{Zhu2007}. Moreover,
$\langle Qk_z^\alpha,k_z^\alpha\rangle\to0$. Hence
\[
\|T_s-Q\|\geqslant |s(\zeta)|.
\]
Since $\zeta\in\mathbb T$ and
$Q\in\mathcal K(A_\alpha^2)$ are arbitrary, it follows that
\[
\|T_s\|_{e,A_\alpha^2}
\geqslant
\max_{\zeta\in\mathbb T}|s(\zeta)|.
\]

For the reverse inequality, define
$s_0\in C(\overline{\mathbb D})$ by
\[
s_0(0)=0,
\qquad
s_0(re^{it})=rs(e^{it}),
\quad 0<r\leqslant1.
\]
Then $s_0=s$ on $\mathbb T$ and
\[
\|s_0\|_\infty
\leqslant
\max_{\zeta\in\mathbb T}|s(\zeta)|.
\]
Thus $s-s_0$ vanishes on $\mathbb T$. By
\cite[Proposition~7.3]{Zhu2007},
\[
T_s-T_{s_0}
=
T_{s-s_0}
\]
is compact on $A_\alpha^2$. Consequently,
\[
\|T_s\|_{e,A_\alpha^2}
=
\|T_{s_0}\|_{e,A_\alpha^2}
\leqslant
\|T_{s_0}\|
\leqslant
\|s_0\|_\infty
\leqslant
\max_{\zeta\in\mathbb T}|s(\zeta)|.
\]
Combining the two estimates gives
\[
\|T_s\|_{e,A_\alpha^2}
=
\max_{\zeta\in\mathbb T}|s(\zeta)|,
\]
to finish the proof.
\end{proof}

\begin{lemma}\label{lem5.6}
Let $\alpha>-1$ and let $a,b\in C(\overline{\mathbb D})$. Then
\[
T_aT_b-T_{ab}\in\mathcal K(A_\alpha^2).
\]
\end{lemma}

\begin{proof}
Let $P_\alpha$ be the Bergman projection on $A_\alpha^2$, and let
$T_a$ and $H_b$ denote the Toeplitz and Hankel operators with symbols
$a$ and $b$, respectively.
For $h\in L^2(\mathbb D,dA_\alpha)\ominus A_\alpha^2$, we have
$
H_{\overline a}^*h=P_\alpha(ah).
$
Hence, for every $f\in A_\alpha^2$,
\[
\begin{aligned}
H_{\overline a}^*H_bf
&=
P_\alpha\bigl(a(I-P_\alpha)(bf)\bigr)\\
&=
P_\alpha(abf)-P_\alpha\bigl(aP_\alpha(bf)\bigr)\\
&=
T_{ab}f-T_aT_bf.
\end{aligned}
\]

Since $b\in C(\overline{\mathbb D})$, the Hankel operator $H_b$ is
compact on the weighted Bergman space $A_\alpha^2$; see
\cite[p.~226]{Zhu2007}. Moreover, $H_{\overline a}^*$ is bounded since
$a\in C(\overline{\mathbb D})$. Consequently,
$H_{\overline a}^*H_b$ is compact, and hence
\[
T_aT_b-T_{ab}\in\mathcal K(A_\alpha^2).
\]
This completes the proof of the lemma.
\end{proof}

\begin{lemma}\label{lem5.7}
Let $\alpha>0$, let $\eta$ be a nonconstant linear fractional self-map
of $\mathbb D$, and let $u$ and $v$ be analytic on a neighborhood of
$\overline{\mathbb D}$. Then
\[
\|M_uD_\eta M_v^*\|_{e,A_\alpha^2}^2
=
\max_{\substack{\xi\in\mathbb T\\ |\eta(\xi)|=1}}
|u(\xi)|^2
|v(\eta(\xi))|^2
|\eta'(\xi)|^{-\alpha},
\]
where the right-hand side is understood to be zero if
$\overline{\eta(\mathbb D)}\subset\mathbb D$.
\end{lemma}

\begin{proof}
If $\overline{\eta(\mathbb D)}\subset\mathbb D$, then $C_\eta$ is
compact on $\mathcal D_\alpha$ by Lemma~\ref{lem4.4}. Hence
$D_\eta$ is compact on $A_\alpha^2$ by Proposition~\ref{prop2.1},
and the conclusion follows immediately.

Suppose that $\eta$ is not an automorphism and
$\overline{\eta(\mathbb D)}\cap\mathbb T\neq\varnothing$.
Since $\eta(\mathbb D)$ is a proper disk internally tangent to
$\mathbb T$, there exist unique points $\xi,\zeta\in\mathbb T$
such that $\eta(\xi)=\zeta$. We first show that
\[
M_{u-u(\xi)}D_\eta\in\mathcal K(A_\alpha^2)\ \ \ \ \mathrm{and}\ \ \ \
D_\eta M_{v-v(\zeta)}^*\in\mathcal K(A_\alpha^2).
\tag{5.3}
\]

For the first assertion, let $\{f_n\}_{n=1}^\infty\subset A_\alpha^2$ converge
weakly to zero. Then $\{f_n\}_{n=1}^\infty$ is bounded. Given $\varepsilon>0$,
choose a neighborhood $U$ of $\xi$ such that
\[
|u(z)-u(\xi)|<\varepsilon,
\qquad z\in U\cap\mathbb D.
\]
Since $D_\eta$ is bounded on $A_\alpha^2$, there exists a constant
$C>0$, independent of $n$ and $\varepsilon$, such that
\[
\int_{U\cap\mathbb D}
|u(z)-u(\xi)|^2
|D_\eta f_n(z)|^2\,dA_\alpha(z)
\leqslant C\varepsilon^2.
\]

Since $\xi$ is the unique boundary contact point of $\eta$, there
exists $0<r<1$ such that
\[
\eta(\mathbb D\backslash U)\subset r\mathbb D.
\]
Indeed, $\eta$ extends continuously to $\overline{\mathbb D}$ and
$|\eta(z)|<1$ on the compact set
$\overline{\mathbb D}\backslash U$.

The sequence $\{f_n\}_{n=1}^\infty$ is bounded and converges uniformly to zero on compact
subsets of $\mathbb D$.
Therefore,
\[
\sup_{z\in\mathbb D\backslash U}|f_n(\eta(z))|
\leqslant
\sup_{w\in r\overline{\mathbb D}}|f_n(w)|
\rightarrow0.
\]
Since $u-u(\xi)$ and $\eta'$ are bounded on $\mathbb D$ and
$D_\eta f_n=(f_n\circ\eta)\eta'$, we have
\[
\begin{aligned}
&\int_{\mathbb D\backslash U}
|u(z)-u(\xi)|^2
|D_\eta f_n(z)|^2\,dA_\alpha(z)
\rightarrow0.
\end{aligned}
\]
Combining the preceding estimates we obtain
\[
M_{u-u(\xi)}D_\eta\in\mathcal K(A_\alpha^2).
\]

We next prove the second assertion. Let $\rho$ be the Krein adjoint of $\eta$. Since $\xi,\zeta\in\mathbb T$ and $\eta(\xi)=\zeta$, the reflection identity gives
\[
\rho^{-1}(\xi)
=\frac{1}{\overline{\eta(\xi)}}
=\eta(\xi)
=\zeta.
\]
Thus $\rho(\zeta)=\xi$. Moreover, $\rho$ is not an automorphism, so $\zeta$ is its unique boundary contact point.

By the Cowen--Hurst formula,
\[
D_\eta^*=M_gC_\rho M_w^*
\]
for some $g$ and $w$ analytic on a neighborhood of
$\overline{\mathbb D}$. Therefore,
\[
M_{v-v(\zeta)}D_\eta^*
=
M_{(v-v(\zeta))g}C_\rho M_w^*.
\]
Set $q=(v-v(\zeta))g$. Since $q(\zeta)=0$ and $\zeta$ is the unique
boundary contact point of $\rho$, the same localization argument as
above gives
\[
M_qC_\rho\in\mathcal K(A_\alpha^2).
\]
Since $M_w^*$ is bounded, it follows that
\[
M_{v-v(\zeta)}D_\eta^*
\in\mathcal K(A_\alpha^2).
\]
Taking adjoints gives
\[
D_\eta M_{v-v(\zeta)}^*
\in\mathcal K(A_\alpha^2).
\]

By \textup{(5.3)}, we have that
\[
M_uD_\eta M_v^*
\equiv
u(\xi)\overline{v(\zeta)}D_\eta
\pmod{\mathcal K(A_\alpha^2)}.
\]
Therefore,
\[
\|M_uD_\eta M_v^*\|_{e,A_\alpha^2}^2
=
|u(\xi)|^2|v(\zeta)|^2
\|D_\eta\|_{e,A_\alpha^2}^2.
\]
By Proposition~\ref{prop2.1} and Lemma~\ref{lem5.4},
\[
\|D_\eta\|_{e,A_\alpha^2}^2
=
\|C_\eta\|_{e,\mathcal D_\alpha}^2
=
|\eta'(\xi)|^{-\alpha}.
\]
Since $\zeta=\eta(\xi)$, we obtain
\[
\|M_uD_\eta M_v^*\|_{e,A_\alpha^2}^2
=
|u(\xi)|^2
|v(\eta(\xi))|^2
|\eta'(\xi)|^{-\alpha}.
\]

It remains to consider the case where $\eta$ is an automorphism.
Put $T=M_uD_\eta M_v^*$. A change of variables shows that
\[
D_\eta^*M_u^*M_uD_\eta=T_r,
\]
where
\[
r(w)
=
|u(\eta^{-1}(w))|^2
|(\eta^{-1})'(w)|^\alpha,
\qquad w\in\mathbb D,
\]
and $T_r$ denotes the Toeplitz operator on $A_\alpha^2$ with symbol
$r$. Indeed, for $f,g\in A_\alpha^2$,
\[
\langle D_\eta^*M_u^*M_uD_\eta f,g\rangle
=
\int_{\mathbb D}
r(w)f(w)\overline{g(w)}\,dA_\alpha(w).
\]
Since $r$ is continuous on $\overline{\mathbb D}$, by Lemma~\ref{lem5.6} we have
\[
T^*T
=
M_vT_rM_v^*
\equiv
T_{|v|^2r}
\pmod{\mathcal K(A_\alpha^2)}.
\]
Hence, by Lemma~\ref{lem5.5} we get that
\[
\begin{aligned}
\|T\|_{e,A_\alpha^2}^2
=
\|T^*T\|_{e,A_\alpha^2}
=
\max_{|w|=1}|v(w)|^2r(w).
\end{aligned}
\]
Since $\eta$ is an automorphism, writing $w=\eta(\xi)$ and using
$
|(\eta^{-1})'(\eta(\xi))|
=
|\eta'(\xi)|^{-1}
$,
we conclude that
\[
\|T\|_{e,A_\alpha^2}^2
=
\max_{\xi\in\mathbb T}
|u(\xi)|^2
|v(\eta(\xi))|^2
|\eta'(\xi)|^{-\alpha}.
\]
This completes the proof.
\end{proof}

We also record a consequence for products with analytic multipliers.

\begin{lemma}\label{lem5.8}
Let $\alpha>0$, let $\varphi$ and $\psi$ be nonconstant linear
fractional self-maps of $\mathbb D$, and let $h$ and $k$ be analytic
on a neighborhood of $\overline{\mathbb D}$. Then
\[
\|M_hC_\varphi C_\psi^*M_k^*\|_{e,\mathcal D_\alpha}^2
=
\max_{\substack{\xi,\zeta\in\mathbb T\\
\varphi(\xi)=\psi(\zeta)\in\mathbb T}}
|h(\xi)|^2|k(\zeta)|^2
\bigl(|\varphi'(\xi)|\,|\psi'(\zeta)|\bigr)^{-\alpha},
\]
where the maximum is understood to be zero when the indexing set is
empty.
\end{lemma}

\begin{proof}
Note that multiplication by an analytic function preserves
$\mathcal D_{\alpha,0}$. Under the unitary map $U_\alpha$, one has
\[
U_\alpha M_hU_\alpha^*g
=
hg+h'\int_0^z g(\omega)\,d\omega.
\]
The integration operator is compact on $A_\alpha^2$, since with respect
to the normalized monomial basis it is a weighted shift whose weights
converge to zero. Hence
$U_\alpha M_hU_\alpha^*\equiv M_h
\pmod{\mathcal K(A_\alpha^2)}$; the same holds for $k$.
Using the finite rank block decomposition and the derivative model,
we therefore obtain
\[
\|M_hC_\varphi C_\psi^*M_k^*\|_{e,\mathcal D_\alpha}
=
\|M_hD_\varphi D_\psi^*M_k^*\|_{e,A_\alpha^2}.
\]

Let $\tau$ be the Krein adjoint of $\psi$. As in the proof of
Lemma~\ref{lem3.1}, the Cowen--Hurst formula gives
\[
D_\varphi D_\psi^*
=
M_uD_{\tau\circ\varphi}M_v^*,
\]
where $u$ and $v$ are bounded invertible analytic multipliers. Using the same direct calculation as
in the proof of Lemma~\ref{lem3.1}, we have
\[
M_hD_\varphi D_\psi^*M_k^*
=
M_{hu}D_{\tau\circ\varphi}M_{kv}^*.
\]
Applying Lemma~\ref{lem5.7} with
$\eta=\tau\circ\varphi$ gives
\[
\begin{aligned}
&\|M_hD_\varphi D_\psi^*M_k^*\|_{e,A_\alpha^2}^2\\
&=
\max_{\substack{\xi\in\mathbb T\\
|(\tau\circ\varphi)(\xi)|=1}}
|h(\xi)u(\xi)|^2
|k((\tau\circ\varphi)(\xi))
v((\tau\circ\varphi)(\xi))|^2
|(\tau\circ\varphi)'(\xi)|^{-\alpha}.
\end{aligned}
\]
If $\zeta=(\tau\circ\varphi)(\xi)\in\mathbb T$, then
$\varphi(\xi)=\psi(\zeta)\in\mathbb T$. Moreover, the formulas for
$u$, $v$, and $\tau'$ obtained in the proof of
Lemma~\ref{lem3.1} yield
\[
|u(\xi)|^2|v(\zeta)|^2
|(\tau\circ\varphi)'(\xi)|^{-\alpha}
=
\bigl(|\varphi'(\xi)|\,|\psi'(\zeta)|\bigr)^{-\alpha}.
\]
Then the desired formula follows.
\end{proof}
\begin{remark}\label{rem5.9}
Relative to the decomposition
$\mathcal D_\alpha=\mathbb C\oplus\mathcal D_{\alpha,0}$, we write
$\widetilde M_h=M_h|_{\mathcal D_{\alpha,0}}$ and similarly for
$\widetilde M_k$. Then
\[
M_h=
\begin{pmatrix}
h(0)&0\vspace{2mm}\\
\Gamma_h&\widetilde M_h
\end{pmatrix},
\qquad
M_k^*=
\begin{pmatrix}
\overline{k(0)}&\Gamma_k^*\vspace{2mm}\\
0&\widetilde M_k^*
\end{pmatrix},
\]
where $\Gamma_h$ and $\Gamma_k$ have rank at most one. Moreover, by
Proposition~\ref{prop2.1},
\[
C_\varphi=
\begin{pmatrix}
I&\Lambda_{\varphi(0)}\vspace{2mm}\\
0&\widetilde C_\varphi
\end{pmatrix},
\qquad
C_\psi^*=
\begin{pmatrix}
I&0\vspace{2mm}\\
\Lambda_{\psi(0)}^*&\widetilde C_\psi^*
\end{pmatrix}.
\]
Consequently,
\[
M_hC_\varphi C_\psi^*M_k^*
=
\begin{pmatrix}
F_{11}&F_{12}\vspace{2mm}\\
F_{21}&
\widetilde M_h\widetilde C_\varphi
\widetilde C_\psi^*\widetilde M_k^*+F_{22}
\end{pmatrix},
\]
where each $F_{ij}$ has finite rank. Hence
\[
M_hC_\varphi C_\psi^*M_k^*
\equiv
\begin{pmatrix}
0&0\vspace{2mm}\\
0&
\widetilde M_h\widetilde C_\varphi
\widetilde C_\psi^*\widetilde M_k^*
\end{pmatrix}
\pmod{\mathcal K(\mathcal D_\alpha)}.
\]

Under the unitary map $U_\alpha$, we have
\[
U_\alpha\widetilde C_\varphi U_\alpha^*=D_\varphi,
\qquad
U_\alpha\widetilde C_\psi^*U_\alpha^*=D_\psi^*,
\]
and
\[
U_\alpha\widetilde M_hU_\alpha^*
\equiv M_h,
\qquad
U_\alpha\widetilde M_k^*U_\alpha^*
\equiv M_k^*
\pmod{\mathcal K(A_\alpha^2)}.
\]
Therefore,
\[
\|M_hC_\varphi C_\psi^*M_k^*\|_{e,\mathcal D_\alpha}
=
\|M_hD_\varphi D_\psi^*M_k^*\|_{e,A_\alpha^2}.
\]
\end{remark}

The next result is a direct consequence of Lemma \ref{lem5.8}.

\begin{theorem}\label{thm5.10}
Let $\alpha>0$, and let $\varphi$ and $\psi$ be nonconstant linear
fractional self-maps of $\mathbb D$. Then
\[
\|C_\varphi C_\psi^*\|_{e,\mathcal D_\alpha}^2
=
\max_{\substack{\xi,\zeta\in\mathbb T\\
\varphi(\xi)=\psi(\zeta)\in\mathbb T}}
\bigl(|\varphi'(\xi)|\,|\psi'(\zeta)|\bigr)^{-\alpha},
\]
where the maximum is understood to be zero when
$\varphi(\mathbb T)\cap\psi(\mathbb T)\cap\mathbb T=\varnothing$.
\end{theorem}

\begin{proof}
This is the special case $h=k=1$ of Lemma~\ref{lem5.8}.
\end{proof}

\begin{theorem}\label{thm5.11}
Let $\alpha>0$, and let $\varphi$ and $\psi$ be nonconstant linear
fractional self-maps of $\mathbb D$. Then
\[
\|C_\psi^*C_\varphi\|_{e,\mathcal D_\alpha}^2
=
\max_{\substack{\lambda\in\mathbb T\\
|\varphi(\lambda)|=|\psi(\lambda)|=1}}
\bigl(
|\varphi'(\lambda)|\,|\psi'(\lambda)|
\bigr)^{-\alpha},
\]
where the maximum is understood to be zero when
$
\varphi^{-1}(\mathbb T)
\cap
\psi^{-1}(\mathbb T)
\cap
\mathbb T
=
\varnothing.
$
\end{theorem}

\begin{proof}
Let $\tau$ and $\sigma$ be the Krein adjoints of $\psi$ and $\varphi$, respectively. Recall from Lemma~\ref{lem4.1} that $C_\psi^*\equiv M_gC_\tau M_h^*$ and $C_\varphi\equiv M_HC_\sigma^*M_G^* \pmod{\mathcal K(\mathcal D_\alpha)}$. By Lemma~\ref{lem4.3} we have
\[
C_\psi^*C_\varphi
\equiv
M_{g(H\circ\tau)}
C_\tau C_\sigma^*
M_{G(h\circ\sigma)}^*
\pmod{\mathcal K(\mathcal D_\alpha)}.
\]
Set
\[
\omega_1=g(H\circ\tau) \ \ \ \ \text{and}\ \ \ \
\omega_2=G(h\circ\sigma).
\]
Applying Lemma~\ref{lem5.8} with
$(\varphi,\psi,h,k)=(\tau,\sigma,\omega_1,\omega_2)$, we obtain
\[
\|C_\psi^*C_\varphi\|_{e,\mathcal D_\alpha}^2
=
\max_{\substack{\xi,\zeta\in\mathbb T\\
\tau(\xi)=\sigma(\zeta)\in\mathbb T}}
|\omega_1(\xi)|^2|\omega_2(\zeta)|^2
\bigl(
|\tau'(\xi)|\,|\sigma'(\zeta)|
\bigr)^{-\alpha}.\tag{5.4}
\]

Suppose that
\[
\tau(\xi)=\sigma(\zeta)=\lambda\in\mathbb T.
\]
By the boundary correspondence between a linear fractional map and
its Krein adjoint, we have
\[
\xi=\psi(\lambda)\ \ \ \ \text{and}\ \ \ \
\zeta=\varphi(\lambda),
\]
and hence
\[
|\psi(\lambda)|=|\varphi(\lambda)|=1.
\]
Conversely, every
$\lambda\in\mathbb T$ satisfying
$|\varphi(\lambda)|=|\psi(\lambda)|=1$
gives such a pair
\[
\xi=\psi(\lambda)\ \ \ \ \mathrm{and}\ \ \ \
\zeta=\varphi(\lambda).
\]
Thus the indexing sets in the two maxima are equivalent.

Moreover, the linear fractional formulas give
\[
|\tau'(\psi(\lambda))|
=
|\psi'(\lambda)|^{-1},
\qquad
|\sigma'(\varphi(\lambda))|
=
|\varphi'(\lambda)|^{-1}.
\]
Since
$\tau(\psi(\lambda))=\sigma(\varphi(\lambda))=\lambda$, we also have
\[
\begin{aligned}
|\omega_1(\psi(\lambda))|^2
|\omega_2(\varphi(\lambda))|^2
&=
|g(\psi(\lambda))H(\lambda)|^2
|G(\varphi(\lambda))h(\lambda)|^2\\
&=
|g(\psi(\lambda))h(\lambda)|^2
|G(\varphi(\lambda))H(\lambda)|^2.
\end{aligned}
\]
Using the definitions of $g,h,G$, and $H$ from
Lemma~\ref{lem4.1}, we obtain
\[
|g(\psi(\lambda))h(\lambda)|^2
=
|\psi'(\lambda)|^{-2\alpha}
\]
and
\[
|G(\varphi(\lambda))H(\lambda)|^2
=
|\varphi'(\lambda)|^{-2\alpha}.
\]
Consequently,
\[
\begin{aligned}
&|\omega_1(\psi(\lambda))|^2
|\omega_2(\varphi(\lambda))|^2
\bigl(
|\tau'(\psi(\lambda))|
|\sigma'(\varphi(\lambda))|
\bigr)^{-\alpha}
\\
&=
|\psi'(\lambda)|^{-2\alpha}
|\varphi'(\lambda)|^{-2\alpha}
\bigl(
|\psi'(\lambda)|^{-1}
|\varphi'(\lambda)|^{-1}
\bigr)^{-\alpha}
\\
&=
\bigl(
|\varphi'(\lambda)|\,|\psi'(\lambda)|
\bigr)^{-\alpha}.
\end{aligned}
\]
Substituting this identity into the formula \textup{(5.4)} proves Theorem~\ref{thm5.11}.
\end{proof}

We finally consider the remaining range $-1<\alpha<0$. In this case,
we use a reduction to a positive Toeplitz operator to obtain a two-sided
estimate for the essential norm.

\begin{lemma}\label{lem5.12}
Let $-1<\alpha<0$, and let $\eta$ be a nonconstant linear fractional
self-map of $\mathbb D$. For every fixed $r>0$,
\[
\|D_\eta\|_{e,A_\alpha^2}^2
=
\|T_{\widehat\mu_{\eta,\alpha}}\|_{e,A_\alpha^2}
\asymp
\limsup_{|z|\to1^-}
\frac{
\displaystyle
\int_{D(z,r)\cap\eta(\mathbb D)}
N_{\eta,\alpha}(w)\,dA(w)
}{
A_\alpha(D(z,r))
},
\]
where $D(z,r)$ is the hyperbolic disk (see Section 2 if necessary) and
\[
d\widehat\mu_{\eta,\alpha}(w)
=
\frac{N_{\eta,\alpha}(w)}{(1-|w|^2)^\alpha}\,dA_\alpha(w).
\]
Here, $T_{\widehat\mu_{\eta,\alpha}}$ denotes the positive Toeplitz
operator determined by
\[
\langle T_{\widehat\mu_{\eta,\alpha}}f,g\rangle_{A_\alpha^2}
=
\int_{\mathbb D}f(w)\overline{g(w)}\,
 d\widehat\mu_{\eta,\alpha}(w).
\]
\end{lemma}

\begin{proof}
Since
\[
dA_\alpha(w)
=
(1+\alpha)(1-|w|^2)^\alpha\,dA(w),
\]
we have
\[
d\widehat\mu_{\eta,\alpha}(w)
=
(1+\alpha)N_{\eta,\alpha}(w)\,dA(w).
\]
For $f,g\in A_\alpha^2$, we have
\[
\begin{aligned}
\langle D_\eta^*D_\eta f,g\rangle_{A_\alpha^2}
&=
(1+\alpha)
\int_{\mathbb D}
f(w)\overline{g(w)}N_{\eta,\alpha}(w)\,dA(w)\\
&=
\int_{\mathbb D}
f(w)\overline{g(w)}\,d\widehat\mu_{\eta,\alpha}(w).
\end{aligned}
\]
Hence
\[
D_\eta^*D_\eta=T_{\widehat\mu_{\eta,\alpha}}.
\]
By the $C^*$-identity and the fact that $\pi$ is a $*$-homomorphism,
\begin{align*}
\|D_\eta\|_e^2
&=
\|\pi(D_\eta)\|^2
=
\|\pi(D_\eta)^*\pi(D_\eta)\|\\
&=
\|\pi(D_\eta^*D_\eta)\|
=
\|D_\eta^*D_\eta\|_e
=
\|T_{\widehat{\mu}_{\eta,\alpha}}\|_e.
\end{align*}

Set $\gamma=\alpha/2$. Then $\gamma>-1/2$ and
$dA_\alpha=(1+\alpha)dV_\gamma$, where
$dV_\gamma(w)=(1-|w|^2)^{2\gamma}dA(w)$. Applying the essential-norm
estimate for positive Toeplitz operators in \cite[Theorem~B]{Yamaji2013}
gives
\[
\|T_{\widehat\mu_{\eta,\alpha}}\|_{e,A_\alpha^2}
\asymp
\limsup_{|z|\to1^-}
\frac{
\displaystyle
\int_{D(z,r)\cap\eta(\mathbb D)}
N_{\eta,\alpha}(w)\,dA(w)
}{
A_\alpha(D(z,r))
}
\]
for every fixed $r>0$.
This proves the lemma.
\end{proof}

\begin{theorem}\label{thm5.13}
Let $-1<\alpha<0$, and let $\varphi$ and $\psi$ be nonconstant linear
fractional self-maps of $\mathbb D$. Let $\tau$ be the Krein adjoint of
$\psi$. Then
\[
\|C_\varphi C_\psi^*\|_{e,\mathcal D_\alpha}^2
\asymp
\limsup_{|z|\to1^-}
\frac{
\displaystyle
\int_{D(z,r)\cap(\tau\circ\varphi)(\mathbb D)}
N_{\tau\circ\varphi,\alpha}(w)\,dA(w)
}{
A_\alpha(D(z,r))
}
\]
for every fixed $r>0$.
\end{theorem}

\begin{proof}
By the finite-rank block decomposition and the derivative model,
\[
\|C_\varphi C_\psi^*\|_{e,\mathcal D_\alpha}
=
\|D_\varphi D_\psi^*\|_{e,A_\alpha^2}.
\]
As in the proof of Lemma~\ref{lem3.1},
\[
D_\varphi D_\psi^*
=
M_uD_{\tau\circ\varphi}M_v^*,
\]
where $u$ and $v$ are bounded invertible analytic multipliers. Hence
\[
\|C_\varphi C_\psi^*\|_{e,\mathcal D_\alpha}
=
\|D_\varphi D_\psi^*\|_{e,A_\alpha^2}
=
\|M_uD_{\tau\circ\varphi}M_v^*\|_{e,A_\alpha^2}
\asymp
\|D_{\tau\circ\varphi}\|_{e,A_\alpha^2}.
\]
Applying Lemma~\ref{lem5.12} with $\eta=\tau\circ\varphi$ gives the desired
result.
\end{proof}

We conclude this section with the following theorem concerning estimates for the essential norm of $C_\psi^*C_\varphi$.

\begin{theorem}\label{thm5.14}
Let $-1<\alpha<0$, and let $\varphi$ and $\psi$ be nonconstant linear
fractional self-maps of $\mathbb D$. Let $\tau$ be the Krein adjoint of
$\psi$. Then
\[
\|C_\psi^*C_\varphi\|_{e,\mathcal D_\alpha}^2
\asymp
\limsup_{|z|\to1^-}
\frac{
\displaystyle
\int_{D(z,r)\cap(\varphi\circ\tau)(\mathbb D)}
N_{\varphi\circ\tau,\alpha}(w)\,dA(w)
}{
A_\alpha(D(z,r))
}
\]
for every fixed $r>0$.
\end{theorem}

\begin{proof}
Let $\sigma$ be the Krein adjoint of $\varphi$. With the notation used
in the proof of Lemma~\ref{lem4.3},
\[
C_\psi^*C_\varphi
\equiv
M_{\omega_1}C_\tau C_\sigma^*M_{\omega_2}^*
\pmod{\mathcal K(\mathcal D_\alpha)},
\]
where $\omega_1$ and $\omega_2$ are bounded invertible analytic
multipliers. Using the finite rank block decomposition and the derivative
model as in Remark~\ref{rem5.9}, we obtain that
\[
\|C_\psi^*C_\varphi\|_{e,\mathcal D_\alpha}
=
\|M_{\omega_1}D_\tau D_\sigma^*M_{\omega_2}^*\|_{e,A_\alpha^2}.
\]
As in the proof of Lemma~\ref{lem4.3}, we have
\[
D_\tau D_\sigma^*
=
M_{\widetilde u}
D_{\varphi\circ\tau}
M_{\widetilde v}^*,
\]
where $\widetilde u$ and $\widetilde v$ are bounded invertible analytic
multipliers. Therefore,
\[
M_{\omega_1}D_\tau D_\sigma^*M_{\omega_2}^*
=
M_{\omega_1\widetilde u}
D_{\varphi\circ\tau}
M_{\omega_2\widetilde v}^*,
\]
and hence
\[
\|C_\psi^*C_\varphi\|_{e,\mathcal D_\alpha}
=
\|M_{\omega_1\widetilde u}
D_{\varphi\circ\tau}
M_{\omega_2\widetilde v}^*\|_{e,A_\alpha^2}
\asymp
\|D_{\varphi\circ\tau}\|_{e,A_\alpha^2}.
\]
Applying Lemma~\ref{lem5.12} with $\eta=\varphi\circ\tau$ completes the
proof.
\end{proof}

\medskip\noindent\textbf{Data availability.} This manuscript has no associated data.

\medskip\noindent\textbf{Conflict of interests.} The authors declare no
competing interests.


\begin{thebibliography}{99}

\bibitem{BourassEtAl2023} M. Bourass, I. Marrhich, and F. Mkadmi,
\newblock Composition operators on weighted analytic spaces,
\newblock {Canad. Math. Bull.}, 66(4) (2023), 1213--1230. MR~4658213.

\bibitem{ChaconChacon2005} G. A. Chac\'on and G. R. Chac\'on,
\newblock Some properties of composition operators on the {D}irichlet space,
\newblock {Acta Math. Univ. Comenianae (N.S.)}, 74(2) (2005), 259--272. MR~2195485.

\bibitem{CliffordLeWiggins2014} J. H. Clifford, T. Le, and A. Wiggins,
\newblock Invertible composition operators: the product of a composition operator with the adjoint of a composition operator,
\newblock {Complex Anal. Oper. Theory}, 8(8) (2014), 1699--1706. MR~3275439.

\bibitem{CliffordLeviNarayan2012} J. H. Clifford, D. Levi, and S. K. Narayan,
\newblock Commutator of composition operators with adjoints of composition operators,
\newblock {Complex Var. Elliptic Equ.}, 57(6) (2012), 677--686. MR~2916827.

\bibitem{CliffordZheng1999} J. H. Clifford and D. Zheng,
\newblock Composition operators on the {H}ardy space,
\newblock {Indiana Univ. Math. J.}, 48(4) (1999), 1585--1616. MR~1757084.

\bibitem{CliffordZheng2003} J. H. Clifford and D. Zheng,
\newblock Composition operators on {B}ergman spaces,
\newblock {Chinese Ann. Math. Ser. B}, 24(4) (2003), 433--448. MR~2024982.

\bibitem{Cowen1988} C. C. Cowen,
\newblock Linear fractional composition operators on {$H^2$},
\newblock {Integral Equations Operator Theory}, 11(2) (1988), 151--160. MR~928479.

\bibitem{CowenMacCluer1995} C. C. Cowen and B. D. MacCluer,
\newblock Composition Operators on Spaces of Analytic Functions,
\newblock Studies in Advanced Mathematics, CRC Press, Boca Raton, FL, 1995. MR~1397026.

\bibitem{CuckovicLe2016} \v{Z}. \v{C}u\v{c}kovi\'{c} and T. Le,
\newblock Adjoints of linear fractional composition operators on weighted {H}ardy spaces,
\newblock {Acta Sci. Math. (Szeged)}, 82(3--4) (2016), 651--662. MR~3616200.

\bibitem{GallardoMontes2003} E. A. Gallardo-Guti\'errez and A. Montes-Rodr\'iguez,
\newblock Adjoints of linear fractional composition operators on the {D}irichlet space,
\newblock {Math. Ann.}, 327(1) (2003), 117--134. MR~2005124.

\bibitem{Hurst1997} P. R. Hurst,
\newblock Relating composition operators on different weighted {H}ardy spaces,
\newblock {Arch. Math. (Basel)}, 68(6) (1997), 503--513. MR~1444662.

\bibitem{LefevreEtAl2013} P. Lef\`evre, D. Li, H. Queff\'elec, and L. Rodr\'iguez-Piazza,
\newblock Compact composition operators on the {D}irichlet space and capacity of sets of contact points,
\newblock {J. Funct. Anal.}, 264(4) (2013), 895--919. MR~3004952.

\bibitem{LefevreEtAl2015} P. Lef\`evre, D. Li, H. Queff\'elec, and L. Rodr\'iguez-Piazza,
\newblock Approximation numbers of composition operators on the {D}irichlet space,
\newblock {Ark. Mat.}, 53(1) (2015), 155--175. MR~3319618.

\bibitem{LengZhao2026} Q. Leng and X. Zhao,
\newblock Schatten class membership of products of composition operators and their adjoints,
\newblock preprint, 2026.

\bibitem{LiLuYu2018} Y. Li, Y. Lu, and T. Yu,
\newblock The essential norms of composition operators on weighted {D}irichlet spaces,
\newblock {Bull. Aust. Math. Soc.}, 97(2) (2018), 297--307. MR~3772661.

\bibitem{LueckingZhu1992} D. H. Luecking and K. Zhu,
\newblock Composition operators belonging to the {S}chatten ideals,
\newblock {Amer. J. Math.}, 114(5) (1992), 1127--1145. MR~1183534.

\bibitem{MacCluerShapiro1986} B. D. MacCluer and J. H. Shapiro,
\newblock Angular derivatives and compact composition operators on the {H}ardy and {B}ergman spaces,
\newblock {Canad. J. Math.}, 38(4) (1986), 878--906. MR~854144.

\bibitem{PauPerez2013} J. Pau and P. A. P\'erez,
\newblock Composition operators acting on weighted {D}irichlet spaces,
\newblock {J. Math. Anal. Appl.}, 401(2) (2013), 682--694. MR~3018017.

\bibitem{PutinarTener2018} M. Putinar and J. E. Tener,
\newblock Singular values of weighted composition operators and second quantization,
\newblock {Int. Math. Res. Not. IMRN}, (2018), no.~20, 6426--6441. MR~3872328.

\bibitem{Shapiro1993} J. H. Shapiro,
\newblock Composition Operators and Classical Function Theory,
\newblock Universitext, Springer-Verlag, New York, 1993. MR~1237406.

\bibitem{Simon2005} B. Simon,
\newblock Trace Ideals and Their Applications, second ed.,
\newblock Mathematical Surveys and Monographs, vol.~120, American Mathematical Society, Providence, RI, 2005. MR~2154153.

\bibitem{Yamaji2013} S. Yamaji,
\newblock Positive {T}oeplitz operators on weighted {B}ergman spaces of a minimal bounded homogeneous domain,
\newblock {J. Math. Soc. Japan}, 65(4) (2013), 1101--1115. MR~3127818.

\bibitem{Zhu2007} K. Zhu,
\newblock Operator Theory in Function Spaces, second ed.,
\newblock Mathematical Surveys and Monographs, vol.~138, American Mathematical Society, Providence, RI, 2007. MR~2311536.

\bibitem{Zorboska1998} N. Zorboska,
\newblock Composition operators on weighted {D}irichlet spaces,
\newblock {Proc. Amer. Math. Soc.}, 126(7) (1998), 2013--2023. MR~1443862.

\end{thebibliography}
\end{document}